\documentclass{conm-p-l}

\usepackage{amssymb}
\usepackage{amsthm}
\usepackage{amsmath}

\usepackage{upref}
\usepackage{cite}

\newcommand{\ol}{\overline}
\newcommand{\wt}{\widetilde}

\newcommand{\tu}{\textup}
\renewcommand{\Re}{\operatorname{Re}}
\newcommand{\mc}{\mathcal}

\newcommand{\mbb}{\mathbb}
\newcommand{\arcsine}{\operatorname{arcsine}}
\newcommand{\cop}{\operatorname{cap}}

\newcommand{\dis}{\operatorname{dis}}
\newcommand{\dist}{\operatorname{dist}}

\renewcommand{\Im}{\operatorname{Im}}

\newcommand{\supp}{\operatorname{supp}}
\renewcommand{\(}{\left( }
\renewcommand{\)}{\right) }
\renewcommand{\[}{\left[ }
\renewcommand{\]}{\right] }

\newtheorem{theorem}{Theorem}[section]
\newtheorem{lemma}[theorem]{Lemma}
\newtheorem{corollary}[theorem]{Corollary}

\theoremstyle{definition}
\newtheorem{definition}[theorem]{Definition}

\theoremstyle{remark}
\newtheorem{remark}[theorem]{Remark}

\numberwithin{equation}{section}

\begin{document}

\title[Orthogonal Polynomials and $S$-curves]{Orthogonal Polynomials and $S$-curves}


%    author one information
\author[E. A. Rakhmanov]{E. A. Rakhmanov}

\address{Department of Mathematics \& Statistics University of South Florida, Tampa, FL 33620-5700, USA}

\email{rakhmano@mail.usf.edu}
\thanks{}

%    author two information
%\author{}
%\address{}
%\curraddr{}
%\email{}
%\thanks{}

\subjclass[2000]{Primary }
%    The 2010 edition of the Mathematics Subject Classification is
%    now available.  If you are citing a classification from the
%    new scheme, use the following input coding instead.
%\subjclass[2010]{Primary }

\date{}

\dedicatory{Dedicated to Francisco (Paco) Marcell\'an on the occasion of his 60th birthday.}

\begin{abstract}
This paper is devoted to a study of $S$-curves, that is systems of curves in the complex plane whose equilibrium potential in a harmonic external field satisfies a special symmetry property ($S$-property).

Such curves have many applications. In particular, they play a fundamental role in the theory of complex (non-hermitian) orthogonal polynomials. One of the main theorems on zero distribution of such polynomials asserts that the limit zero distribution is presented by an equilibrium measure of an $S$-curve associated with the problem if such a curve exists. These curves are also the starting point of the matrix Riemann-Hilbert  approach to srtong asymptotics. Other approaches to the problem of strong asymptotics (differential equations, Riemann surfaces) are also related to $S$-curves or may be interpreted this way.

Existence problem $S$-curve in a given class of curves in presence of a nontrivial external field presents certain challenge. We formulate and prove a version of existence theorem for the case when both the set of singularities of the external field and the set of fixed points of a class of curves are small (in main case -- finite). We also discuss various applications and connections of the theorem.
\end{abstract}

\maketitle

\tableofcontents
%\newpage

\section{Complex orthogonal polynomials and $S$-curves}\label{sec1}

\subsection{Introductory Example}\label{subsec1.1}

There are two types of orthogonal polynomials. The more usual and older type of orthogonality is hermitian one with respect to a positive weight (measure). Associated polynomials present, in particular, basis in weighted Hilbert spaces and used for polynomial approximation. In the last few decades another type of orthogonal polynomials -- the complex (non-hermitian) ones with analytic weights came to attention. They appear first of all as denominators of Pad\'e approximants and other kinds of ``free poles rational approximations''. Such polynomials are related, for instance, to continued fractions and three terms recurrence relations. They satisfy certain model differential equations. Both types are also related to their particular boundary value problems and to their own equilibrium problems. Equilibrium problems related to complex orthogonal polynomials is the main topic of the paper.

We begin with an exampe of orthogonality with varying weights of the real line. In this case the two types of orthogonality coinside and  associated orthogonal polynomials have all the properties mentioned above. 

Let $\Phi_n(x)$ be a sequence of continuous real-valued functions on $\Gamma=[-1,1]$, $f(x)>0$ a.e.\  on $\Gamma$. Let
\begin{equation}
\label{eq1.1}
f_n(x)=e^{-2n\Phi_n(k)}f(x),\quad x\in\Gamma
\end{equation}
and polynomials $Q_n(x)=x^n+\dotsc$ are defined by orthogonality relations
\begin{equation}
\label{eq1.2}
\int_\Gamma Q_n(x)\ x^k~f_n(x)\,dx=0,
\quad k=0,1,\dotsc,n-1.
\end{equation}
The following theorem by A.\  A.\  Gonchar and the author was probably the first general result on zero distribution of orthogonal polynomials with varying (depending on the degree of the polynomial) weight; see \cite{GoRa84}.

% T 1.1
\begin{theorem}\label{thm1.1}
If $\Phi_n(x)\to\Phi(x)$ uniformly on $\Gamma$ then
$$
\frac1n\,X\(Q_n\)=\frac{1}{n}\sum_{Q_n(\varsigma)=0}
\delta(\varsigma)\ \  {\overset{*}{\to}}\ \  \lambda
$$
where $\lambda=\lambda_\phi$ is the equilibrium measure of $\Gamma$ is the external field $\Phi$.
\end{theorem}

In terms of the (total) energy
\begin{equation}
\label{eq1.3}
\mc{E}_\varphi(\mu)=\iint\log\frac1{(x-y)}\,
d\mu(x)\,d\mu(y)+2\int\Phi(x)\,d\mu(x),
\end{equation}
the equilibrium measure is defined by the minimization property in class $\mc{M}(\Gamma)$ of all unit positive Borel measures $\mc{M}$ on $\Gamma$
\begin{equation}
\label{eq1.4}
\mc{E}_\varphi(\lambda)=\min_{\mc{M}\in\mc{M}(\Gamma)}\mc{E}_\varphi(\mu).
\end{equation}
The actual theorem in \cite{GoRa84} was more general. It was formulated for $\Gamma=\mbb{R}$ (which requires certain conditions on $\Phi(x)$ as $x\to\pm\infty$), convergence of $\Phi_n(x)$ to $\Phi(x)$ was of a weaker type. It was also noticed that $\Gamma$ may be replaced by a system of curves in complex plane if standard hermitian orthogonality $\int_\Gamma Q_n(x)\ol{x}^kf(x)|dx|=0$ in \eqref{eq1.2} is assumed. Finally, it is possible to consider complex valued $f(x)$: the theorem is still valid under some assumption on $|f(x)|$ and $\arg~f(x)$.

Situation changes if $\Phi(x)=\varphi(x)+i\wt{\varphi}(x)$ is complex valued. In this case we have first of all assume that $f(x)$, $\Phi_n(x)$ and, therefore, $\Phi(x)$ are analytic in a domain $\Omega\supset(-1,1)$ (otherwise the problem is illposed). Then $\Gamma=[-1,1]$ in \eqref{eq1.2} may be replaced by any rectifiable curve $\Gamma$ in $\Omega$ connecting points $-1,1$ (existence of integral near those points is assumed). Then orthogonality relation \eqref{eq1.2} are preserved; orthogonality become non-hermitian (complex).

It turns out that assertion of theorem \ref{thm1.1} remains valid (with $\Phi$ replaced by $\varphi=\Re\Phi$) if there exists a curve $S$  in  class $\mc T$ of curves $\Gamma\subset\Omega$ connecting $-1,1$ with a special symmetry property ($S$-property) for its total potential
$$
V^\mu(z)+\varphi(z)=\int\log\frac1{|z-x|}\,d\mu(x)+\varphi(x)
$$
with external field $\varphi$.  This is a simple particular case of a general theorem \ref{thm1.3}  below. 
We note that the equilibrium measure $\lambda=\lambda_{\varphi,\Gamma}$ for a compact $\Gamma$ in the external field $\varphi$ defined by \eqref{eq1.3} is equivalently defined by the following relation in terms of total potential
\begin{equation}
\label{eq1.5}
\begin{split}
\(V^\lambda+\varphi\)(x)
&=w,~x\in\supp\lambda \\
&\ge w,~x\in\Gamma
\end{split}
\end{equation}
(Equation \eqref{eq1.5} uniquely defines pair of measure $\lambda\in\mc{M}(\Gamma)$ and constant $w=w_{\varphi,\Gamma}$ -- equilibrium constant.)

%\textbf{Definition 1.2}
\begin{definition}\label{def1.2}
Let $S$ be a compact in $\mbb{C}$ and $\varphi$ be a harmonic function in a neighborhood of $S$. We say that $S$ has an $S$-property relative to external field $\varphi$ if there exist a set $e$ of zero capacity such that for any $\zeta\in S\sim e$ there exist a neighborhood $D=D(\zeta)$ for which $\supp(\lambda)\cap D$ is an analytic arc and, furthermore, we have
\begin{equation}
\label{eq1.6}
\frac{\partial}{\partial n_1}\(V^\lambda+\varphi\)(\zeta)
=\frac{\partial}{\partial n_2}\(V^\lambda+\varphi\)(\zeta)\quad \zeta\in\supp(\lambda)\sim e
\end{equation}
where $\lambda=\lambda_{\varphi,S}$ is the equilibrium measure for $S$ in $\varphi$ and $n_1,n_2$ are two oppositely directed normals to $S$ at $\zeta\in S$ (we can actually admit that  $\varphi$ has a small singular set included in $e$). 
\end{definition}

So, one of central for this paper concepts is defined by the pair of conditions \eqref{eq1.5} --  \eqref{eq1.6}  (with $\Gamma = S$). 
In particular, it follows by \eqref{eq1.5} that distribution of a positive charge presented by $\lambda$ is in the state of equilibrium on the fixed conductor $S$. 
 On the other hand,  $S$-property of compact $S$  in \eqref{eq1.6}  in electrostatic terms means that forces acting on element of charge at $\zeta$ from two sides of $S$ are equal. So, the equilibrium distribution  $\lambda$ of an $S$-curve will remain in equilibrium if we remove the condition that the charge belongs to $S$ and make the whole plane a conductor (except for a few insulating points -- endpoints of some of arcs in support of $\lambda$). Thus, $\lambda$ presents a distribution of charge which is in equilibrium in condacting domain; such an equilibrium is unstable. 

In terms  of energy \eqref{eq1.4} the $S$property  \eqref{eq1.6} of equilibrium measure is equivalent to the fact the $\lambda$ is a a critical point of weighted energy functional with respect tolocal variations. We will go into further details in sec. \ref{sec4} below; technically speaking, variations of equilibium energy is one of two fundamental components in the proof of the main result of the paper.

This was, in short, electrostatic characterization of the limit zero disributions of complex orthogonal polynomials. 

%  Equilibrium measure on a conductor (compact) $\Gamma$ defined in \eqref{eq1.4} or, equivalently, in \eqref{eq1.5} presents distribution of a positive  charge with minimal energy (so, it is extremal measure).

% Sec. 1.2
\subsection{General theorem on zero distribution of complex orthogonal polynomials}\label{subsec1.2}

Let $\Omega$ be a domain in $\mbb C$, $S$ be a compact in $\Omega$. Further, let $\Phi_n(z)\in H(\Omega)$ and $\Phi_n(z)\to\Phi(z)$ uniformly on compacts in $\Omega$ as $n\to\infty$. Finally, let $f\in H(\Omega\sim S)$ and polynomials $Q_u(z)=z^n+\dotsb$ are defined by orthogonality relations with weights $f_n=f e^{-2n\phi_n}$
\begin{equation}
\label{eq1.7}
\oint_S Q_n(z)z^k f_n(z)dz=0,\quad k=0,1,\dots,n-1;
\end{equation}
where integration goes over the boundary of $\ol{\mbb C} \sim S$ (if such itegral exist, otherwise integration goes over an equivalent cycle in $\ol{\mbb C} \sim S$). The following is a complex version of theorem \ref{thm1.1}: see \cite{GoRa87}.

%\textbf{Theorem 1.3}
\begin{theorem}\label{thm1.3}
 If $S$ has $S$-property in $\varphi=\Re\Phi(z)$ and complement to the support of equilibrium measure $\lambda=\lambda_{\varphi,S}$ is connected then\quad $\frac1n\,X\(Q_n\)\overset{*}{\ \to\ }\lambda$.
\end{theorem}

The last assertion is equivalent to convergence
\begin{equation}
\label{eq1.8}
\left|\wt{Q}_n(z)\right|^\frac{1}{n}
\overset{\cop}{\ \to\ }\exp\left\{-V^\lambda(z)\right\}
\end{equation}
in capacity in $\mbb{C}\sim\supp\lambda$ for spherically normalized polynomials $\wt{Q}_u=c_nQ_n$. In addition we have
\begin{equation}
\label{eq1.9}
\left|\oint_S\wt{Q}^2(t)f_n(t)\frac{dt}{t-z}\right|^\frac12
\overset{\cop}{\ \to}\  e^{-2w}\ ,\quad z\in\mbb{C}\sim\supp\lambda
\end{equation}
where $w=w_{\varphi,S}$ is the extremal constant for $S$.

% Sec. 1.3
\subsection{The existence problem for $S$-curves}\label{subsec1.3}

In a typical application of theorem \ref{thm1.3} in approximation theory  function (element) $f$ is analytic and multi-valued in $\Omega\sim e$ where   $e\subset\Omega$ is a small set, that is, element $f$ has analytic continuation along any path in the domain. We denote a class of such functions by $\mc{A}(\Omega\sim e)$. We also denote by $\mbb P_n$ set of polynomials of degree at most $n$.

For $f\in\mc{A}(\Omega\sim e)$ let $\mc{T}=\mc{T}_f$ be a set of systems of curves $\Gamma$ in $\Omega$ such that $f\in H(\Omega\sim\Gamma)$. Let  a system of polynomials $Q_n(z)\in\mbb{P}_n$ is defined by \eqref{eq1.7} with integration over $\Gamma\in\mc{T}$ (in place of $S$). Important is that  generally $S\in\mc{T}$ with $S$-property is not given and its existence is not known. 

This leads ro the problem of finding $S\in\mc{T}$ with $S$-property. The context may be different from what we have used as the original motivation;  the problem may not be related to complex orthogonal polynomials. In general, we have a domain $\Omega$ with a harmonic function external field in it (may be, with a small set of singular points) and a class $\mc{T}$ of curves in $\Omega$. We want to find a curve with $S$-property in $\mc{T}$ if such a curve exists. 
%Later in this paper we will discuss a general version of the problem. Now we return to our original context.

%We note that orthogonality relations in \eqref{eq1.7} are more general than in \eqref{eq1.2}  (consider them over same set $\Gamma$). If $\Gamma$ is smooth and $f(z)$ has integrable boundary values $f^+$  and $f^-$ on $\Gamma$ (actually we need $f\in H_1$) then we have
%$$\oint\limits_\Gamma(*)\,f\,dz = \int\limits_\Gamma(*)\ w\,dz; \quad w=f^+-f^-$$
%which returns us back to \eqref{eq1.2} (with $f(x)$ in \eqref{eq1.1} replaced with $w(x)$). In general $f$ may have non-integreble singular points in $e$ and contour integral may bypass these singularities. 

Fundamentally important is the case $\varphi\equiv 0$. First, $S$-existence problem in this case essentially includes a number of extremal problems in geometric theory of analytic functions. Second, the case is related to classical convergence problem for (diagonal) Pad\'e approximants to functions with branch points. We mention some of related results in the next section.

At the same time the case $\varphi\equiv 0$ is rather specific in context of $S$-existence problem. Under general assumptions on class $\mc{T}$, an $S$-curve exists and is unique \cite{St85}. In the presence of a nontrivial external field, neither existence nor uniqueness are guaranteed.

% Sec. 2

\section{Pad\'e approximants for functions with branch points}\label{sec2}

For a finite set ${A}=\left\{a_1,\dotsc,a_p\right\}\in\mbb{C}$ of distinct points we consider an element at $\infty$
\begin{equation}
\label{eq2.1}
f(z)=\sum^\infty_{n=0}\frac{f_k}{z^k}\in\mc{A}(\mbb{C}\sim{A}).
\end{equation}
Let $\pi_n(z)=(P_n/Q_n)(z)$ be diagonal Pad\'e approximants to $f$, that is polynomials $P_n,Q_n\in\mbb{P}_n$ ($\mbb{P}_n$ --- the set of all polynomials of degree at most $n$) are defined by
\begin{equation}
\label{eq2.2}
R_n(z):=\(Q_nf-P_n\)(z)=O\left(\frac1{z^{n+1}}\right),\quad z\to\infty
\end{equation}
(see \cite{BaGM81-1},  \cite{BaGM81-2}  for details).

First results on convergence of the sequence $\left\{\pi_n\right\}$ to $f$ (for ${A}\not\subset\mbb{R}$) were obtained by J.\  Nuttall who also made the following conjecture (see \cite{Nu77}, \cite{Nu84}). Let
% (2.3)
\begin{equation}
\label{eq2.3}
\mc{T}=\{\Gamma\subset\mbb{C}:f\in H(\mbb{C}\sim\Gamma)\}
\end{equation}
and $S\in\mc{T}$ is defined by the minimal capacity property
% (2.4)
\begin{equation}
\label{eq2.4}
\cop(S)=\min_{\Gamma\in\mc{T}}\cop(\Gamma)
\end{equation}
Then sequence $\left\{\pi_n\right\}$ converges to $f$ in capacity in the complement to
$$
\pi_n\overset{\cop}{\to}f,\quad z\in\mbb{C}\sim S
$$
The conjecture has been proven by H. Stahl \cite{St85},\cite{St86} under more general assumption $\cop(A)=0$. More exactly, he proved the following

%\textbf{Theorem 2.1}
\begin{theorem}\label{thm2.1}
Let $\cop(A)=0$ then
\begin{enumerate}
\item[(i)] there exist and unique $S\in\mc{T}_f$  with $S$-property for which $w=f^+-f^-\not\equiv 0$ on any analytic arc in $S$; 

\item[(ii)] for denominator $Q_n$ of Pad\'e approximants we have $\frac1n\,\mc{X}\(Q_n\)\overset{*}{\ \to\ }\lambda$ where $\lambda $ is equilibium measure for $S$.
\end{enumerate}
\end{theorem}

Note that (ii) implies  \eqref{eq1.8} and \eqref{eq1.9} with $f_n \equiv 1$. H. Stahl created an original potential theoretic method for studying limit zero distribution of Pad\'e denominators $Q_n(z)$ based directly on orthogonality relations
$$
\oint_\Gamma Q_n(z)\,z^k\,f(z)\,dz=0,\quad
k=0,1,\dotsc,n-1;\quad\Gamma\in\mc{T}_f.
$$
for these polynomials written with $\Gamma = S$ where $S$ ihas property \eqref{eq1.6}.  The method was further developed in \cite{GoRa87} for the case of presence nontrivial external field; Theorem \ref{thm1.3} has been proved using this devepopment. 
% which works if the class $\mc{T}_f$ contains a compact with an $S$-property.

 It is important to observe that according to assertion \tu{(i)}  of the theorem an compact $S$-compact always exists  (there is no ``if'' in the theorem). For a finite set $A$ this part of the theorem is close to well-known Chebotarev's problem in geometric function theory.

Chebotarev's problem was the problems of existence and characterization of a continuum of minimal capacity containing $A$. It was solved independently by Gr\"otzsch and Lavrentiev in the 1930s. The following theorem is just one example of a large class of theorems presenting solutons of extremal problems in function theory in terms of quadratic differentials \tu{(}for details we refer to \cite{Ku80}\tu{) and \cite{St84}}.

%\textbf{Theorem 2.2}
\begin{theorem}\label{thm2.2}
For a given set ${A}=\left\{a_1,\dotsc,a_p\right\}$ of $p\ge2$ distinct points in $\mbb{C}$ there exist a unique set $S$
$$
\cop(S)\ =\ \min_{\Gamma\in\mc{T}}\cop(\Gamma)
$$
where $\mc{T}$ is the class of continua $\Gamma\subset\mbb{C}$ with ${A}\subset\Gamma$. The complex Green function $G(z)=G(z,\infty)$ for $\ \ol{\mbb{C}}\sim S$ is given by
$$
G(z)=\int^z_a\sqrt{{V(t)}/{A(t)}}\,dt,
\quad V(z)=z^{p-2}+\dotsb\in\mbb{P}_{p-2}.
$$
where $ A(z) = (z-a_1) \dots(z-a_p\ )$ and $V$ is uniquely defined by $A$. Thus,  we have
$$
S=\{z:\Re G(z)=0\}
$$
so that $S$ is a union of some of critical trajectories of quadratic differential $(V/A)(dz)^2$. It is also the zero level of green function $g(z)=\Re G(z)$ of two-sheeted Riemann surface for $\sqrt{V/A}$.
\end{theorem}

Stahl proved that asserions of theorem \ref{thm2.2} remain valid for $S$ in theorem \ref{thm2.1}. In his recent paper \cite{St06} he introduced a more general concept of ``extremal cuts'' \tu{(}or maximal domain\tu{)} related to an element $f$ of an analytic function at $\infty$. Let $f\in\mc{A} \(\ol{\mbb{C}}\sim e\)$ where $e$ is a compact of positive capacity. Let $\mc{T}_f=\mc{T}$ is defined by \eqref{eq2.3} above. The main result in \cite{St06} is the existence and uniqueness theorem of a compact $S\subset\mc{T}_f$ with \eqref{eq2.4}.In general such $S$ does not have an $S$-property \tu{(}but contains a part with this property\tu{)}. The paper contains also an extended review related to the standard case $\cop(e)=0$.

% Section 3

\section{An existence theorem for an $S$-curve in harmonic external field.}\label{sec3}

We formulate conditions of existence in terms of Hausdorff metric and local variations; terms are explained in Sections \ref{subsec3.1} and \ref{subsec3.2} next.

% Sec. 3.1
\subsection{Hausdorff metric}\label{subsec3.1}
\hfill \\
Classes of compacts. Let $d\(z_1,z_2\)=\left|z_1-z_2\right|\(1+\left|z_1\right|^2\)^{-1/2}\(1+\left|z_2\right|^2\)^{-1/2}$ be the chordal distance in $\ol{\mbb{C}}$. For two compacts $K_1,K_2\subset\ol{\mbb{C}}$ their Hausdorff distance $\delta_H$ is defined as
\begin{equation}
\label{eq3.1}
\ol{\delta}_H\(K_1 K_2\)=\inf\left\{\delta>0:K_1\subset
\(K_2\)_{\delta},K_2\subset\(K_1\)_{\delta}\right\}
\end{equation}
where $(K)_{\delta}-\delta$--neighborhood of $K$ in the chordal metric:
\begin{equation}
\label{eq3.2}
(K)_{\delta}=\left\{z\in\ol{\mbb{C}}:\min_{\zeta\in K}
d(z,\zeta)<\delta\right\}.
\end{equation}

In what follows we consider mainly compacts which are finite unions of continua. For such compacts $K$ we denote by $s(K)$ the number of connected components of $K$.

% $s(K)=s$ means that $K$ is a union of $\le s$ disjoint continua.

Let $\Omega$ be a domain in $\ol{\mbb{C}}$, then set of all compacts $K\subset\ol{\Omega}$ is a compact in $\ol\delta_{H}$-metric. The same is true for all compacts $K\subset\ol{\Omega}$ with $s(K)\leq s$

We use chordal metric on $\ol{\mbb C}$ and associated Hausdorff metric $\ol{\delta}_H$ only to formulate the theorem \ref{thm3.1} below. In the proof of the theorem in Sec. \ref{sec9} we immediately reduce consideration of compacts $K\subset\ol{\mbb C}$ to the case $K\subset\mbb C$; so that we can use Euclidean distance $d(z_1,z_2)=|z_1 - z_2|$,  associated neighbourhoods $(K)_\delta$ and corresponding Hausdorff distance $\delta_H$. Metrics are equivalent on compacts in any fixed disc in the open plane.

%Sec. 3.2
\subsection{Local variations}\label{subsec3.2}

For a closed set $A\subset\ol{\mbb{C}}$  and a complex valued functions
$$
h(z)\in\mbb{C}^1\(\ol{\mbb{C}}\)\quad\text{ with }\quad h(z)=0,\quad z\in\mc{A}.
$$
We introduce one parametric family of local variations with fixed set $A$ in the direction of $h$ - transformations of $\ol{\mbb{C}}$ defined by
\begin{equation}
\label{eq3.3}
z\to z^t=z+th(z),\quad t\ge0
\end{equation}
We also call them  $A$-variations  (in the direction of $h$). For small enough $t$  transformations  \eqref{eq3.3} are one-to-one; in a stadard way they generate variations of compacts $K^t$ and measures $\mu^t$
\begin{equation}
\label{eq3.4}
K\to K^t = \{z^t: z\in K\}, \quad  d\mu^t(z^t) = d\mu(z)
\end{equation}

Variations defined above leave the point $\infty$ unchanged; it is always technically convenient to  include infinity into the fixed set. In many cases we are interested in variations of a particular compact  $K$, then the condition $h(z)\in\mbb{C}^1\(\ol{\mbb{C}}\)$ may be replaced by $h(z)\in\mbb{C}^1\(\Omega\)$ where $\Omega$ is a domain containing $K$.
It is important to note that in many cases we use more restricted classes of A-variations; in particular,  we often use functions $h$ satisfying a stronger condition $h=0$ in a neighbourhood of $A $ (see remark \ref{rem4.1}).
%
% Subsection 3.3
%
\subsection{Main Theorem}\label{subsec3.3}

The following is a version of an existence theorem.

%\textbf{Theorem 3.1}
\begin{theorem}\label{thm3.1}
Suppose that the external field $\varphi$ and class of curves $\mc{T}$ satisfy the following conditions:
\begin{enumerate}
\item[(i)]  $\varphi$ is harmonic in $\ol{\mbb{C}}\sim{e}$ where ${e}\subset{\ol{\mbb{C}}}$ is a finite set

\item[(ii)]
\begin{enumerate}
\item[(a)]  $\mc{T}$ is closed in a $\ol{\delta}_{H}$-metric

\item[(b)] For a finite set ${A}\subset{\ol{\mbb{C}}}$  $\quad\mc{T}$ is open in topology of $A$-variations:\newline
$\Gamma\in\mc{T}$ implies $\Gamma^t\in\mc{T}$ for small enough $t$ for any $A$-variation.

\item[(c)]  For some natural $s$ any $\Gamma\in\mc{T}$ has at most  $s$ connected components;
\end{enumerate}

\item[(iii)]
\begin{enumerate}
\item[(a)]  There exist $\Gamma\in\mc{T}$ with\newline
$\mc{E}_{\varphi}[\Gamma]=\inf\limits_{\mu\in\mc{M}(\Gamma)}E_{\varphi}(\mu)>-\infty$;

\item[(b)]  $\mc{E}_{\varphi}\{\mc{T}\}=\sup\limits_{\Gamma\in\mc{T}}E_{\varphi}[\Gamma]<+\infty$.
\end{enumerate}

\item[(iv)]  For any sequence $\Gamma_n\in\mc{T}$ that converges in $\ol{\delta}_H$ metric \tu{(}$\ol{\delta}_H(\Gamma_n,\Gamma)\to 0$\tu{)} there exists a disc $D$ such that $\Gamma_n\sim D\in\mc{T}$
\end{enumerate}

Then there exists a compact $S\subset\mc{T}$ with the $S$-property.

Moreover, for the equilibrium measure $\lambda=\lambda_{\varphi,S}$ of the compact $S$ the following condition is satisfied
\begin{equation}\label{eq3.5}
R(z)=\Big(\int\frac{d\lambda(t)}{t-z}+\Phi'(z)\Big)^2\in H(\Omega\sim A\cup e)
\end{equation}
that is, the function $R(z)$ is holomorphic (analytic and single valued) in $\Omega\sim(e\cup A)=\Omega' \ $ ($ \varphi = \Re \Phi$). 

Furthermore, $R$ has simple poles at each point $a\in A\sim e$ (singularities at points in $e$ are essentially those of $(\Phi')^2$.)

Finally, $supp\ \lambda$ us a union of analytic arcs which are trajectories of quadratic differential - $R(z)(dz)^2$ and this differential is the one with closed trajectories; we also have $d\lambda(z)=\frac{1}{\pi}|\sqrt{R}dz|$ along arcs in $supp\ \lambda$.

\end{theorem}
% Comments on Theorem 3.1
A short outline of the proof is presented in Section \ref{subsec3.4}; details are contained in Section \ref{sec9}. Particular versions of the theorem and parts of related techniques were discussed in \cite{GoRa87,PeRa94,KaRa05,MaRa11}. Here we make a few general remarks.

Points of sets $e$ and $A$ are both singularities of $R(z)$ in \eqref{eq3.5} above; but points from $A\sim e$ are simple poles of $R$, singularities of $R(z)$ at points in $e$ are same as in $\Phi'(z)^2$. Besides that there is not much difference between sets $e$ and $A$ - both will eventually be fixed points of variations and will often unify them in single set $A$. 

Theorem \ref{thm3.1} remains valid if both sets $A$ and $e$ are closed sets of capacity zero. Case of finite sets is, probably, the most interesting, since $R(z)$ in this case has a finite number of singular points in $\ol{\mbb C}$ and in many cases may be explicitly found. On the other hand sest of sigularities of positive capacity present an essential problem which we do not discuss in this paper.

In the connection with basic properties (ii) of class of curves $\mc T$ we note that associated assumptions may be stated in different terms. We have selected an abstract way which has probably some advantage in generality. 

The following  example is fundamental. Let $f\in\mc A\ (\ol{\mbb C}\sim A)$ and let $\mc T_f$ be the class \eqref{eq2.3} of admissible cuts $\Gamma$ associated with $f$. In many applications the class $\mc{T}$ may be directily or indirectly defined as $\mc{T}_f$ with some $f\in\mc A\ (\ol{\mbb C}\sim A)$.

It is important to observe that $S$-curve exists in $\mc T_f$ but the class does not formally satisfy conditions of the theorem \ref{thm3.1}. It is not closed in $\ol{\delta}_H$ -metric since the definition contain an implicit condition that the point where element $f$ is defined (in the current case it is $\infty$) is not in any $\Gamma$. The situation is typical; many important classes of curves are not closed.  However, in most of naturally arising situations the maximizing sequence converges to an element which belong to $\mc T_f$. Moreover, it is usually not difficult to prove it;  we will have a few examples below. 

Alternatively we may define $\mc T$ as collection of finite unions $\Gamma=\cup\Gamma_j$ of continua $\Gamma_j$; each of them connects certain groups of points in $A$ or separates one group from others. Such classes $\mc T$ would satisfy conditions of the theorem if we do not have exceptional points). In this connection we note that some particular $S$-problems (especially with rational external field) are similar to  "extremal partitions" or "modulii"  problems in geomtric function theory; see \cite{Ku80}, \cite{St84} and also  \cite{MaRa11}.

Theorem \ref{thm3.1} will not be valid without condition (iic) since any compact may be approximated in $\ol{\delta}_H$-metric by a finite number of points. Continuity of equilibrium energy is not preserved in such circumstances.

We present a few examples related to conditions (iii). Let  $\mc T$ be class of continua connecting points  $0$  and  $1$  and  $\varphi = \log z $ then $S$-compact does not exist since $\mc{E}_{\varphi}[\Gamma] = - \infty$ for any $\Gamma\in\mc{T}$ (in connection with possible applications it is more reasonable to include curves $\Gamma$ bypassing zero; then (iiia) is satisfied and $S$-curve exists). In another typical example  $\mc T$ is set of  Jordan contours separating zero and infinity $\varphi =  a\log z,\ a>0 $.  If $a\ne 1/2$ we have $\mc{E}_{\varphi}\{\mc{T}\} = \infty$ so that the condition (iiib) is not satisfied and $S$-curve does not exist (if $a < 1/2$ a maximizing sequence $\Gamma_n$ collapses to zero; if $a>1/2$ it collapses to $\infty$). If $a=1/2$ then any circle cenred ar zero is an $S$-curve which is an example of non uniqueness.

Condition (iv) is technical. Haussdorf closure always contains large sets which are of no interest in the connection with maximization for equilibrium energy. For instance, let $\mc T$ be the family of analytic arcs containing two fixed points. Then its closure $\ol{\mc T}$ in Haussdorf metric contains $\Gamma=\ol{\mbb C}$. Without condition (iv) such possibilities will require unnecessary separate consideration.

\subsection{Outline of the proof: $\max$-$\min$ energy method}\label{subsec3.4}

The method has two components. First, we study continuity properties of the equilibrium energy functional
\begin{equation}
\label{eq3.6old}
\mc{E}_\varphi[\Gamma]=\inf_{\mu\in\mc{M}(\Gamma)}\mc{E}_\varphi(\mu):\mc{T}\to[-\infty,+\infty]
\end{equation}
on $\mc{T}$ with the Hausdorff metric. The main result here is the following

%\textbf{Theorem 3.2}
\begin{theorem}\label{thm3.2}
Functional $\mc{E}_\varphi$ above is upper semi-continuous.
\end{theorem}

We present the proof of the theorem in Section \ref{sec9} where we introduce a larger class $\mc K$ of compacts $K$ in $\mbb{C}$ with a bounded number of connected components. First we consider continuous external field $\varphi$ (not supposed to be harmonic) and prove that the functional
$$
\mc{E}_\varphi:\mc{K}\to(-\infty,+\infty)
$$
is continuous. Then we use continuous fields $\varphi$ to approximate fields with singular points.

Theorem \ref{thm3.2} is actually the most extended part of the proof; the rest of it is comparetively easy reduced to known results. An immediate  corollary of the theorem is the following
%\textbf{Corollary 3.3}
\begin{corollary}\label{cor3.3}
There exist the extremal compact $s$ with
\begin{equation}
\label{eq3.6}
\mc{E}_\varphi[S]=\sup_{\Gamma\subset\mc{T}}\mc{E}_\varphi[\Gamma];\quad S\in\mc{T}.
\end{equation}
\end{corollary}

The second component of the method is a method of energy variations developed in  \cite{PeRa94}, \cite{KaRa05}  and  \cite{MaRa11}  which lead to the following

%\textbf{Theorem 3.4}
\begin{theorem}\label{thm3.4}
The compact $S$ in \eqref{eq3.6} above has $S$-property.
\end{theorem}

The only part of the theorem \ref{thm3.4} which require a proof is the following.

%\textbf{Lemma 3.5}
\begin{lemma}\label{lem3.5}
Equilibrium measure of extremal compact $S$ in \eqref{eq3.6} is a critical measure associated with external field $\varphi$ and fixed set $A$.
\end{lemma}

The fact that total potential of a critical measure satisfies $S$-property is known. We go into some details in the next section. Thus, lemma \ref{lem3.5} concludes the proof of existence theorem.

% Section 4
\section{Critical measures and equilibrium measures of $S$-curves }\label{sec4}

\subsection{Main definitions}\label{subsec4.1}

Let   $\Omega\subset\mbb{C}$ be a domain,  $A\subset\Omega$ be a finite set and  $\varphi$ be a harmonic function in $\Omega\sim A$. For an  $A$-variation $z\to z^t=z+th(z)$ in $\Omega\sim A$ we define the associated variation of weighted energy a measure $\mu$ in $\Omega$ by
% 4.1
\begin{equation}
\label{eq4.1}
D_h\mc{E}(\mu)=\lim_{t\to 0+}\frac{1}{t}\(\mc{E}_\varphi\(\mu^t\)-\mc{E}_\varphi(\mu)\).
\end{equation}

We say that $\mu$ is $(A,\varphi)$-crucial if for any $A$-variation such that the limit above exists we have
\begin{equation}
\label{eq4.2}
D_h\mc{E}(\mu)=0
\end{equation}

There are two important facts about critical measures.  First, the equilibrium measure of an $S$-curve is a critical measure. Second, the potential of a critical measure $\lambda$ has $S$-property \eqref{eq1.6}. Besides, critical measures are often more convenient to deal with then $S$-curves.
 
Critical (stationary) measures were first introduced in \cite{GoRa87} and then used in \cite{PeRa94} and later in\cite{KaRa05} in combination with min-max method. A systematic study of critical measures for rational field was carried out in \cite{MaRa11}. The last paper contains all the facts related to critical measures which we need here (a few cosmetic changes required). We single out some important formulas leading to $S$-property of a critical measure. First, we have the following explicit representation \cite{MaRa11}.
% 4.3
\begin{equation}\label{eq4.3} 
D_h\mc{E}(\mu)=\ \Re\bigg(\iint\frac{h(x)-h(y)}{x-y}d\mu(x)d\mu(y)-2\int\Phi'(x)h(x)d\mu(x)\bigg)
\end{equation}
where $\varphi(z)=\Re\Phi(z)$; we have $\Phi'\in H(\Omega\sim{A})$. 

Next, $\lim$ in \eqref{eq4.1} exists iff $\Phi'h\in L_1(\mu)$. For technical reasons it make sense to replace the last condition with a more restrictive one $h(z)=0$ for $z\in(A)_\delta$ with some $\delta>0$ and it is convenient to include it in the definition of critical measure (see Remark \ref{rem4.1}).
% 4.4
Next, for any $({A},\varphi)$ critical measure $\mu$ we have
\begin{equation}\label{eq4.4}
R(z)=\(\int\frac{d\mu(t)}{t-z}+\Phi'(z)\)^2\in H(\Omega\sim{A})
\end{equation}
(lemmas 2 and 3 in \cite{MaRa11}). In turn, this assertion implies the following description of $\Gamma=\supp\mu$: $\Gamma$ consists of a finite number of critical or closed trajectories of quadratic differential $-R(dz)^2$. Moreover, the last differential is closed (more exactly, it is quadratic differential with closed trajectories \cite{St84}), \cite{MaRa11} Theorem 5.1. Together with \eqref{eq4.4} this  yields representation
% 4.5
\begin{equation}\label{eq4.5}
V^\mu(z)+\varphi(z)=-Re\int^z_a\sqrt{R(t)}\,dt,\quad z\in\mbb{C}\sim\Gamma
\end{equation}
The $S$-property of $\Gamma$  together with the formula $d\mu(z)=\frac1{\pi}|\sqrt{R}\,dz|$ on open arc of $\Gamma$ follow directly from \eqref{eq4.5}

%\textbf{Remark 4.1}
\begin{remark}\label{rem4.1}
Actually, to derive properties \eqref{eq4.4} and \eqref{eq4.5} of a critical measure one needs to verify \eqref{eq4.2} for some particular class of functions $h$ . More exactly, it was shown in \cite{MaRa11}  that it is enough to verify \eqref{eq4.2} for modified Shiffer's variations $h(z)=\theta(z){A(z)}/{\(z-\zeta\)}$ where $A$  is the polynomial whose roots are fixed points, $\zeta$ is a complex parameter and  $\theta(z)$ is a real function which is equal to zero in a neighbourhood of points from $\mc A$ and it is equals to unity outside of a slightly larger neighborhoods. 
We do not have to go into such details; instead  we  define critical measures as a measures $\mu$ with \eqref{eq4.2} valid for any smooth $h(z)$ 
satisfying the condition $h(z)=0$ for $z\in( A)_\delta$ with some $\delta>0$. This class is larger but easier to describe. 

%Condition $h=0$ on  $\mc A$  is enough if $\varphi$ has only logarithmic singularities. Otherwise  $h$  has to be modified near each singular point $a\in S_\mu$ to satisfy the condition $h\Phi'\in L_1(\mu)$ in case of more generalsingular points. In \cite{MaRa11} a modfication was considered made by multiplying $h={A}/{\(z-z_0\)}$ by a function $\theta(z)$ with is $\theta=0$ in a neighbourhood of points from $\mc A$ and  $\theta=1$ outside of slightly larger neighborhoods. Such modification works well for the purposes of this paper, however, we do not have to go into details which can be found in \cite{MaRa11}. So we will define critical measures as a measures $\mu$ with \eqref{eq4.1} valid for any smooth $h(z)$ satisfying the condition $h(z)=0$ for $z\in(e)_\delta$ with some $\delta>0$.
 
\end{remark}

We note that class of (unit positive) critical measure is larger than class of equilibrium measure of $S$-compact (for the same external field and same fixed set). Indeed, for the equilibrium measure $\lambda=\lambda_{\varphi,S}$ we have relations \eqref{eq1.5} and \eqref{eq1.6}; in particular, $V^\lambda+\varphi$ is a constant on the $\supp\lambda$ (besides, there is a part $S\sim\supp\lambda$ of $S$ which has in general certain degree of freedom).
Thus, we have two closely relate notions
\begin{enumerate}
\item[(1)]  $S$-curve in a (homotopic) class $\mc T$ with fixed set $A$ in the external field $\varphi$

\item[(2)]  $(A,\varphi)$-critical measure $\mu$
\end{enumerate}
Next we make some more comments in this direction for the (classical) case $\varphi\equiv 0$; see \cite{MaRa11} for details.

% Sec.4.2
\subsection{$A$-critical measures and $S$-curves for $\varphi\equiv 0$ }\label{subsec4.2}
We compare critical measures and equilibrium measures of $S$-curves without external field and with the same fixed set $A=\left\{a_1,\dotsc,a_p\right\}$.
Let classes of curves be introduced as $\mc{T}_f$ associated with various functions  $f\in \mc A (\ol{\mbb{C}}\sim A)$ . For any (unit positive) $A$-critical measure there exists a polynomial 
%$V(z)=\underset{v=1}{\overset{p-2}{\Pi}}(z-v_j)$ 
$\  V(z) = \prod_{j=1}^{p-2} (z-v_j)$ such that with $\ A(z) = \prod_{k=1}^{p} (z-a_k)$  (here we denote polynomial and its set of zeros by same symbol) we have
% 4.6
\begin{equation}
\label{eq4.6}
V^\mu(z)=\Re\int_{a_1}^{z}
\sqrt{{V(t)}/{A(t)}}\,dt,\quad
d\mu(z)=\frac{1}{\pi}\left|\sqrt{{V}/{A}}\,dz\right|.
\end{equation}
(in particular, $\supp\mu$ is a union of critical trajectories of $\({V(z)}/{A(z)}\)(dz)^2$). The same is true for  the equilibium measures of $S$-curves in class  $\mc{T}_f$. Thus, both sets of mesures may be characterized in terms of associated polynomials  $V$. Now, using zeros $v_j$ of $V$ as parameters,  we descibe set of A-crtical measures (\cite{MaRa11}, sec. 9).

 Space of vectors $\mc{V}_A=\{ v = \(v_1,\dotsc,v_{p-2}\)\}$ corresponding to critical measures has real dimension $p-2$; more exactly, it is a union of $3^{p-2}$ bordered (bounded) domains (cells) on a manyfold of real dimention $p-2$ . Interior points of each cell correspond to measures $\mu$ with $\Gamma=\supp\mu$ consisting of exactly $p-1$ simple disjoint analytic arcs $\Gamma_j$ with endpoints from  $\{a_k,\  v_j\}$. As parts of ${\mbb C}^{p-2}$ cells (corresponding $v\in{\mbb C}^{p-2} $) are defined by systems of eqations
% 4.7
\begin{equation}
\label{eq4.7}
\Re\int_{\Gamma_j}\sqrt{{V(t)}/{A(t)}}\,dt=0,\quad j=1,\dotsc,p-2
\end{equation}

On the other hand 
%we could consider many different classes of $\mc{T}$ of curves with fixed set  $A$, but 
there is only a finite number of  $S$-curves  associated with a given fixed set $A$ and any of them may be obtained as $S$-curves associated with a class $\mc{T}_f$ of admissible cut for a properly selected functions $f\in\mc{A}(\ol{\mbb{C}}\sim A)$. Equilibrium measures of $S$-curves are among ${A}$-critical measures and they are located on boundaries of cells. In particular, $v^0$ corresponding to the Chebotarev's continuum $\mc T=\mc T^0_A$ for $A$ belongs to the boundary of each cell (and may be used for a more explicit description of $\mc V_A$). Further, equilibrium measures of $S$-curves satisfy \eqref{eq4.7} and also $p-2$ additional equations which distinguish them among critical measures. These additional equations may be presented as follows. For any critical measure $\mu$ with disjoint arcs $\Gamma_j$ of support of $\mu$ we have
$$
V^\mu(z)=C_j,\quad z\in\Gamma_j \quad  (\supp\mu=\bigcup\limits_{j=1}^{p-1}\Gamma_j )
$$
 Additional equations for an equilibrium measure of an $S$ compact are 
$$C_1=\dotsb=C_{p-1}$$ 
These equations may also be written in the form \eqref{eq4.7} so,  all of them are equtions on real parts of periods of quadratic differential $V/A(dz^2)$.
We have totally $2p-2$ real equations for the same number of real parameters in $V$. This system of equations defining equilibrium measures of $S$-curves has rather complicated analytic structure. Critical measures may be useful in this context since they consitute a connected space.

% Sec. 4.3
\subsection{Heine--Stieltjes and generalized Jacobi polynomials}\label{subsec4.3}

Critical measures were studied in \cite{MaRa11} with the purpose of characterization of zero distributions of Heine--Stieltjes polynomial. For a fixed $A(z)=\prod\limits_{k=1}^{P}(z-a_k)$ with distinct $a_k$ and $B(z)={\alpha z}^{p-1}+\dotsm\in\mbb{P}_{p-1}$ and a fixed $n\in N$ there exist $\delta_n= \dbinom{n+p-1}{n}$ polynomials $V_n(z)\in\mbb{P}_{p-2}$ (Van Vleck polynomials) such that differential equation
\begin{equation}\label{eq4.8}
A(z)y''(z)+B(z)y'(z)-n(n+\alpha-1)V_n(z)y(z)=0
\end{equation}
has a polynomial solution $y(z)=Q_n(z)=z^n\dots$ of degree $n$ (Heine--Stieltjes polynomials). One of the main results in \cite{MaRa11} is the following. If we have a convergent sequence of Van Vleck polynomials $V_n(z)\to V(z)=z^{p-2}+\dotsb$ then for corresponding Heine--Stieltjes polynomials we have
$$
\frac1n\,\mc{X}\(Q_n\)\ \overset{*}{\to}\ \mu
$$
where $\mu$ is an $A$-critical measure. Moreover any $A$-critical measure may be obtained this way.  In electrostatic terms $\frac1n\,\mc{X}\(Q_n\)$ is a discrete critical measure and the result simply means that a weak limit of a sequence of discrete critical measures is a (continuous) critical measure.

It is interesting to compare Heine-Stieltjes polynomials with Generalized Jacobi polynomials --  denominators of diagonal Pade approximants at infinity  for functions
$$
f(z)=\exp\(\frac12\int^z\frac{B(t)}{A(t)}dt\)
=\prod_{k=1}^{p}\(z-a_k\)^{\alpha_k}
$$
(assume that $\alpha_1+ \dots + \alpha_k=0$; then $f$ is analytic at $\infty$). Let $P_n(z)$ be the corresponding Pad\'e denominator. It is a classical fact that for each $n\in\mbb{N}$ there exists a polynomial $h_n(z)=z^{p'-2}+\dotsb\in\mbb{P}_{p'-2}$ where $p'\le p$ such that $y=P_n(z)$ is a solution of thed Laguerre equation
% 4.9
\begin{equation}\label{eq4.9}
Ah_n y''+\(A'h_n-Ah'_n+Bh_n\)y'-n(n+1)C_ny=0
\end{equation}
where $C_n(z)=z^{2p'-4}+\dots\in\mbb P_{2p'-4}$. Comparison of \eqref{eq4.8} and \eqref{eq4.9} shows that $P_n(z)$ are also Heine--Stieltjes polynomials with redefined  $A$ and $B$. 

Note that by Stahl's theorem we have $\frac1n\,\mc{X}\(P_n\)\overset{*}{\to}\lambda$ where $\lambda$ is the Robin measure for $S$-compact $S\in\mc{T}_f$. Thus, Heine--Stieltjes and generalized Jacobi polynomials present discrete versions of critical measures and equilibrium measures of $S$-curves respectively.

Strong asymptotics for GJ polynomials $P_n$ based on the equation \eqref{eq4.9} were obtained by J. Nuttall \cite{Nu86} for $p=3$; recently the result was extended for arbitrary $p$ \cite{MaRaSu11}.. Strong asymptotics for HS polynomials $Q_n$ based  on the equation \eqref{eq4.9} were obtained in  \cite{MaRa10}.  Related formulas may be interpreted as solution of an $S$-problem; we will go into some details in Section \ref{sec8} below.

% Section 5
\section{Rational external fields}\label{sec5}
In section \ref{subsec4.2} above we discussed  equilibrium measures of $S$-cueves and critical measures associated with a fixed set $A=\left\{a_1,\dotsc,a_p\right\}$ in a classical nonweighted case. Now we introduce an important class of rational external fields and single out a few corollaries of theorem \ref{thm3.1}.

% Sec. 5.1
\subsection{Field of a system of fixed charges in plane}\label{subsec5.1}

 We place a (real) charge $\alpha_k$ at the fixed point $a_k$. This system of cahrges generate the external field $\varphi(z)=\sum\limits_{k=1}^{p}\alpha_k\log(1/{\left|z-a_k\right|})$. Then corresponding $\Phi'=-\sum\frac{\alpha_k}{z-a_k}$ is rational function (so that we consider the case as rational).

Let $\mc{T}_f$  be class of curves associated with a function $f\in\mc{A}\(\ol{\mbb{C}}\sim A\)$ (recall that in such situations we admit that $ \Gamma\in\mc{T}$ may bypass $a_k$ with negative $\alpha_k$).  Theorem \ref{thm3.1} implies that there is a compact $S_f\in\mc{T}_f$ with $S$-property. Subsequently it follows by \eqref{eq4.4} that corresponding $R$ is a rational function with second order poles at points from $A$ and seconds order zero at $\infty$. So, there is a polynomial $V(z)=\prod\limits_{k=1}^{2p-2}\(z-v_k\)$ such that for the total potential of equilibrium measure $\lambda$ we have
% 5.1
\begin{equation}
\label{eq5.1}
\(V^\lambda+\varphi\)(z)=\Re\int_{a_1}^{z}\frac{\sqrt{V(t)}}{A(t)}\,dt;
\quad d\lambda(z)=\frac1{\pi}\left|\frac{\sqrt{V(z)}}{A(z)}\,dz\right|
\end{equation}
The same is true for any $(A,\varphi)$ critical measure $\mu$. Thus, both families of all (unit, positive) critical measures and equilibrium distributions of $S$-compacts in the field $\varphi$ are described as subsets of space $\mbb C^{2p-2}= \{ v=\(v_1,\dotsc,v_{2p-2}\)\}$  just as for the case $\varphi\equiv 0$ in Section \ref{sec4} above.

Furthermore, similar to what we have in case $\varphi\equiv 0$, the set of $(A,\varphi)$-critical measure is a union of cells --- bordered manyfolds $M^{p-2}_j$ of real dimention $p-2$ which may be locally represented using ${v}\in\mbb{C}^{2p-2}$  as local parameters by a system of equations including $p-2$ real equations
\begin{equation}
\label{eq5.2}
\Re\int_{\Gamma_j}{\sqrt{V(t)}}/{A(t)}\,dt=0,\quad j=1,\dotsc,p-2
\end{equation}
and $p$ complex equation
\begin{equation}
\label{eq5.3}
{\sqrt{V(a_j)}}/{A'(a_j)}=\alpha_j,\quad j=1,\dotsc,p
\end{equation}
(\eqref{eq5.3} prescribes residues of ${\sqrt{V}}/{A}$ at poles).
We formally state what we said above as theorem \ref{thm5.1} below (joint result with A. Mart{\'\i}nez-Finkelshtein).

%\textbf{Theorem 5.1}
\begin{theorem}\label{thm5.1}
Let $\varphi(z)=\sum\limits_{k=1}^{p}\alpha_k\log(1/{\left|z-a_k\right|})$, $\alpha_k\in\mbb{R}$. Then set of $(A,\varphi)$-critical measures is a union of finite number of cells - bordered manyfolds of real dimension $p-2$. Interior points of each cell correspond to measures $\mu$ whose support is a union of $p-1$ disjoint analytic arcs $\Gamma_j$  Endpoints of those arcs belong to $\left\{a_1,\dotsc,a_p, v_1,\dotsc,v_{2p-2}\)$ and satisfy equations \eqref{eq5.2} and \eqref{eq5.3}.
% An explicit form for $\mu$ is given by \eqref{eq5.1}.
\end{theorem}
% $\supp\mu = \bigcup\limits_{\int=1}^{p-1}\gamma_j$ 
Extra $p-2$ equations which distinguish equilibrium measures of $S$-curves in various classes ${\mc T}_f$ are similar to those in section \ref{sec4}. The key point in the proof of the theorem is existence of weighted Chebotarev's continuum associated with the $\varphi$ which is the compact $S$ with $S$-property is class $\mc{T}^0$ of continua containing $A$ in the external field $\varphi$ (see theorem \ref{thm3.1}).

As a corollary of the theorem \ref{thm5.1} one can obtain extentions of the theorem on zero distribution of Heine-Stieltjes polynomials \eqref{eq4.8} from \cite{MaRa11} discuused in \ref{subsec4.3} to a larger class with $B=B_n$ depending on $n$ in such a way that there exist a limit $\underset{n\to\infty}{\lim}\frac{1}{n}B_n(z)$. Such polynomials are interesting in the connection with so called Gaudin's model (see \cite{MutaVa08}; see also \cite{MaSa02} for an earlier "real" version of the theorem). Subsequently, methods of \cite{MaRa10} may be used to obtain strong asymptotics for those polynomials.

%  Sec.5.2
\subsection{Multi-point Pad\'e approximants}\label{subsec5.2}

As another application of theorem \ref{thm3.1}, we consider a problem of interpolation of several (generally different) analytic elements by a single rational function with free poles. Interpolation points are defined by the set $A=\left\{a_1,\dotsc,a_p\right\}\subset\mbb{C}$, then for each $k=1,\dots,p$  we introduce a finite set sets of distinct points $A_k\subset\ol{\mbb{C}}$ (branch points of functons) ; sets  $A_k$ may intersect. Next, let  $f_k$ be a function elements at $a_k$ with branch points from $A_k$
$$
f_k(z)=\sum_{n=0}^\infty c_{k,n}\(z-a_k\)^n\in\mc{A}\(\ol{\mbb{C}}\sim A_k\),\quad k=1,\dotsc,p$$
Finally, let $m_{k,n}\in\mbb{Z}_+$, $k=1,\dotsc,p$, $n\in\mbb{N}$ be nonnegative integers satisfying conditions
$$
\sum_{k=1}^pm_{k,n}=2n+1,\quad
\lim_{n\to\infty}\frac{m_{k,n}}n=m_k,
\quad k=1,\dotsc,p
$$
(so that $m_1+\dotsb+m_p=2$). Then we can find a rational function $\pi_n=P_n/Q_n$ of order $\le n$ interpolating $f_k$ at $a_k$ with multiplicity $m_{k,n}$. Polynomials $P_n, Q_n$ are defined by conditions
$$
\(Q_nf_k-P_n\)(z)=O\(\(z-a_k\)^{m_{k,n}}\)\text{ as }z\to a_k
$$
(such polynomials $P_n,Q_n\in\mbb{P}_n$ exist and the ratio $P_n/Q_n$ is unique).

The problem of convergence and the sequence $\left\{\pi_n\right\}$ reduces to the problem of zero distribution for polynomials ${Q}_n$. Those polynomials may be defined by orthogonality conditions with weights including varying parts  $1/\Omega_n(z)$ where $\Omega_n(z)=\prod\limits_{k=1}^n\(z-a_k\)^{m_{k,n}}$.

Investigation of the zero distribution of $Q_n$ requires (but is not completely reduced to) the existence of an $S$-curve in the class $\mc{T}$ of curves $\Gamma_0\cup\Gamma_1\cup\dotsb\cup\Gamma_p\subset\ol{\mbb{C}}\sim A$ with the following properties:
\begin{gather*}
\ol{\mbb{C}}\sim\Gamma_0=\Omega_1\cup\Omega_2\cup\dotsb\cup\Omega_p;
\quad a_k\in\Omega_k;\quad\Omega_i\cap\Omega_j=\emptyset; \\
\Gamma_k\subset\Omega_k,\quad f_k\in H\(\Omega_k\sim\Gamma_k\)
\end{gather*}
and the additional condition that the jump of $f$ over any analytic arc in $\Omega$ is $\not\equiv 0$.

Now, there are certain exceptional situations when conditions of theorem \ref{thm3.1} are not satisfied and the required $S$-curve does not exists. For instance it happen when we interpolate two differents constants at two different points (see example in section \ref{subsec3.3}). We do not go into furher details; normally conditions of the theorem \ref{thm3.1} are satisfied, $S$-curve exists and may also be defined by the extremal property
$ \mc{E}_\varphi[S]=\max_{\Gamma\in\mc{T}}\mc{E}_\varphi[\Gamma] $
with the external field $\varphi(z)=\sum\limits_{k=1}^pm_k\log\left|z-z_k\right|$.

However, in many cases support of the equilibrium meaure related to the problem is disconnected and theorem \ref{thm1.3} is not directly applied. A way around was found in  \cite{BuMaSu}  for a particular case of two-point approximation  for two elements with two branch points each. We note also that  if the poles of the interpolating function $\mc{F}_n$ (zeros of $Q_n$) are partially fixed then we come to a problem more general then the one above; external field has an additional positive charge which makes situation a little more complicated. 

% Sec 5.3
\subsection{Polynomial external field}\label{subsec5.3}

Let $p(z)=cz^m+\dotsb$ be a polynomial of degree $m\ge 2$. Let $\mc{T}$ be a class of curves $\Gamma=\Gamma_1\cup\Gamma_2\cup\dotsb\cup\Gamma_p$ where $\Gamma_j$ is a Jordan arc in $\ol{\mbb{C}}$ which goes from $\infty$ to $\infty$. More exactly, each $\Gamma_j$ begins and ends in different sectors where $\varphi(z)=\Re p(z)\to\infty$ in such a way that the integrals in
$$
\int_\Gamma Q_n(z)z^ke^{-2np(z)}dz=0,\quad k=0,1,\dotsc,n-1
$$
exist (and not trivially vanish for any polynomials $\mc{Q}_n$.) Then the relations above define a sequence of polynomials $Q_n(z)=z^n+\dotsb$. Then, by Theorem \ref{thm1.3}, we have $\frac1n\,\chi\(Q_n\)\to\lambda$, where $\lambda$ is the equilibrium measure of the $S$-curve in $\varphi$ with $S\in\mc{T}$ if such a curve exists.

Conditions of theorem \ref{thm3.1} are not satisfied; indeed, set of singularities (which is also a fixed set) consists of a single point $\infty$ but class  $\mc{T}$ is not closed (for instance, $\Gamma=\{\infty\}\not\in S$ is a limit point of $\mc{T}$ in the spherical $\delta_H$ metric). However, it is not difficult to show that for any $\Gamma^{(n)}\to\Gamma\not\in\mc{T}$ we have $\mc{E}_\varphi\(\Gamma^{(n)}\)\to-\infty$ and, therefore, maximizing $\Gamma^{(n)}$ has a limit point in $\Gamma$. By Theorem \ref{thm3.1}, this limit point has the $S$-property.

An interesting problem is to characterize polynomials $R(z)$ which may be obtained this way and, thus, characterize corresponding $S$-curves. Some progress  in this direction was made in \cite{Be08}. One parametric family $p(z)=tz^2+z^4$ is considered in \cite{BeTo11}. The work \cite{MaOrRa} is in progress where equilibrium measures on $\mbb{R}$ in the field of $\varphi=\Re p$, $p\in\mbb{P}_4$ are considered. The problem is now far from being completely solved.

We return to the existence problem for external field $\varphi=\Re p$: There is another way to arrange the reference to Theorem \ref{thm3.1} which may have an independent interest. In short, we can cut off curves $\Gamma=\cup\Gamma_j\in\mc{T}$ taking intersection with a disc of large radius and subsequently take the limit as radius tends to infinity. More exactly, each $\Gamma_j$ is defined by two different sectors of plane. Let $\theta_1,\theta_2$ be arguments of middle lines of those sectors. Define $a_{j,n}=ne^{i\theta_1}$, $b_{j,n}=ne^{i\theta_2}$: and let $\Gamma_{j,n}$ be any curve in plane connecting $a_{j,n}$ and $b_{j,n}$. Let $\Gamma_n$ be class of curves $\Gamma_n=\cup\Gamma_{j,n}$; this class satisfies condition of Theorem \ref{thm3.1} and there exists an $S$-compact $S_n\in\mc{T}_n$. For its equilibrium measure $\lambda_n$ we have
$$
R_n(z)=\(\int\frac{d\lambda_n(t)}{t-z}+p'(z)\)^2=\frac{q_n(z)}{s_n(z)}
$$
where $s_n(z)=\sum\limits_j\(z-a_{j,n}\)\(z-b_{j,n}\)$ and $q(z)\in\mbb{P}_{m'}$ where $m'=2m-2+2p$. We note that $S$-curves and measure $\lambda_n$ associated with the situation have independent interest if endpoints are prescribed in advance. In the original context, as those points tends to $\infty$ the function $R_n$ will become a polynomial as $n\ge N$ ($q_n$ is divisible by $S_n$) and corresponding $S_n$ maybe completed arcs going to $\infty$ to a curve $S$ of original problem.

The situation discussed above is typical: in many applications conditions of theorem 3.1 are not immediately satisfied, but reduction is eventually possible.

% Section 6
\section{Green's and the vector equilibrium problems}\label{sec6}

Each equilibrium problem for logarithmic potential has an analogue for the Green's potential. Most part of corresponding definitions and assertions  are essentially identical. We make a brief review of them next.

\subsection{Definitions}\label{subsec6.1}
Let $\Omega$ be a domain and $\mu$ be a positive Borel measure in $\Omega$. We denote by
\begin{equation}
\label{eq6.1}
V_\Omega^\mu(z)=\int g(z,t)\,d\mu(t),\quad z\in\Omega
\end{equation}
the Green's potential of $\mu$ ($g(z,\zeta)$ is the Green function for $\Omega$ with a pole at $z=\zeta$). This class of potentials is in many ways simpler than class of logarithmic potential. The function $V=V_\Omega^\mu$ is invariant under conformal mappings of $\Omega$ (it may be defined as solution of Poisson equation $\Delta V=-2\pi\mu$ with boundary values $V=0$ on $\partial\Omega$). We have for any positive $\mu$ in $\Omega$
$$
V^\mu_\Omega(z)\ge 0;\quad
\mc{E}^\Omega(\mu)=\int V^\mu_\Omega d\mu\ge 0.
$$
For a compact $F\subset\Omega$ and an external field $\varphi\in C(F)$ there exist a unique equilibrium measure $\lambda$ on $F$ defined by each of the conditions \eqref{eq1.4} and \eqref{eq1.5} where $\mc{E}_\varphi$ and $V^\lambda$ have to be replaced by $\mc{E}_\varphi^\Omega$ and $V^\lambda_\Omega$.

Next, let a $\varphi$ be a harmonic function in $\Omega\sim e$ ($\cop(e)=0$). We say that $F$ has Green's $S$-property (relative to $\Omega$) and $\varphi$) if \eqref{eq1.6} is valid for $V^{\lambda_F}_\Omega$ in place of $V^\lambda$. Problem of existence of a compact with (Green's) $S$-property we will abbreviate to (Green's) $S$-problem.

As for logarithmic potential, the case $\varphi=0$ is classical. It is customary in this case to set $E =\ol{\mbb{C}}\sim\Omega$ and consider a pair of disjoint compacts (condenser) $(E,F)$ instead of compact $F$ in the domain $\Omega=\ol{\mbb{C}}\sim E$. Capacity $C(E,F)$ of a condenser $(E,F)$ is defined by
$$
\cop\(E,F\)=\frac1{\mc{E}^\Omega(\lambda_F)}
=\frac1{\mc{E}\(\lambda_F-\lambda_E\)}
$$
where $\Omega=\mbb{C}\sim E$, $\lambda_F$-Green's equilibrium measure for $F$ (with $\varphi\equiv 0$); $\lambda_E$ is balayage of $\lambda_F$ onto $\partial\Omega\subset E$. Signed measure $\lambda=\lambda_F-\lambda_E$ is called equilibrium measure of condenser $(E,F)$. We note that $E$ and $F$ are interchangeable in this context, in particular $C(E,F)=C(F,E)$. In the presence of an external field we assume that it is acting on $F$.

We note also that $\lambda=\lambda_F-\lambda_E$ may also be interpreted as a vector measure $\vec{\lambda}=\(\lambda_E,\lambda_F\)$; we will go into some detail later in Section \ref{subsec6.4}.

The equilibrium problems for Green's potential (S-problem, in particular) play, first of all, an important role in the theory of best rational approximations to analytic functions. Next we review briefly classical results by J. Walsh, A. Gonchar and H. Stahl related to the case $\varphi=0$.

% Subsection 6.2
\subsection{Gonchar's $\rho^2$-conjecture}\label{subsec6.2}
Let $E\subset\ol{\mbb{C}}$ be continuum and $f=f_E$ be an element of analytic function on $E$. Let $\mbb{R}_n$ be the set of all rational functions $r_n=P_n/Q_n$ of order $\le n$ $\(P_n,Q_n\in\mbb{P}_n\)$ and
$$
\rho_n(f)=\min_{r\in\mbb R_n}\max_{z\in E}|f(z)-r(z)|.
$$
A well-known theorem of the 1930s proposed by J. Walsh asserts that
\begin{equation}
\label{eq6.2}
\varlimsup_{n\to\infty}
\rho_n(f)^{\frac1n}
\le\rho(f)=\inf_{F\in\mc{F}}e^{-1/C(E,F)}.
\end{equation}
Where $\mc{F}$ is the set of all compacts $F\subset\ol{\mbb{C}}\sim E$ such that $f$ has analytic continuation from $E$ to $\Omega=\ol{\mbb{C}}\sim F$ (see \cite{Wa60} for further details.)

In 1978,  A. Gonchar \cite{Go78} (see also \cite{GoLo78}) studied approximations for Markov-type functions $f$ on an interval $E$ of the real axis and observed that it is possible in this case to replace $\rho(f)$  in Walsh's estimate with $\rho(f)^2$. Later, he generalized this result and proved also that in many important cases $\varlimsup$ may be replaced with $\lim$ and $\le$ may be changed to $=$ (see \cite{Go78} for details and further references). Thus, for some important classes of functions equality holds
\begin{equation}
\label{eq6.3}
\lim_{n\to\infty}\rho_n(f)^{\frac1n}=\rho(f)^2
\end{equation}

So, Gonchar made his so-called called $\rho^2$-conjecture: let for some compact $e\subset\ol{\mbb{C}}\sim E$, $\cop(e)=0$ function $f|_E$ has analytic continuation along any path in $\ol{\mbb{C}}\sim e$. Then \eqref{eq6.3} is satisfied.  The upper bound in this conjecture (with $\varlimsup$) was proved in 1986 by Stahl \cite{St86}. The associated low bound and, therefore, existence of $\lim$ was established a year later in \cite{GoRa87}. So, the $\rho^2$-conjecture becomes a theorem.

Proof of the Gonchar--Stahl theorem was based on a general version of multi-point Pad\'e approximants (free poles interpolation). In short it may be presented as follows: Let $\mc{F}_f$ be the class of compact $F\subset\Omega$ with $f\in H\(\ol{\mbb{C}}\sim F\)$. There exists a compact $S\subset\mc{F}_f$ with Green's $S$-property relative to $\Omega$ \cite{St86}; let $\lambda_F$ be its Green's equilibrium measure. Finally, let $\lambda_E$ be the balayage of $\lambda_F$ onto $\partial\Omega\subset E$.

Let $G_n(z)=z^{2n}+\dotsb\in\mbb{P}_{2n}$ be a sequence of polynomials with zeros on $E$ such that $\frac{1}{2n}\mc X(G_n)\overset{*}{\to}\lambda_E$ as $h\to\infty$. Let $\pi_n(z)=P_n/ Q_n\in\mbb{R}_n$ be corresponding sequence of (linear) Pad\'e approximants to $f$ that is $(Q_nf-P_n)/ G_n$ is analytic on $E$.

Then denominators $Q_n$ of approximations $\pi_n$ satisfy orthogonality relations
% 6.4
\begin{equation}
\label{6.4}
\oint_{F}Q_n(z)z^k\frac{f(z)}{G_n(z)}\,dz=0,\quad k=0,1,\dotsc,n-1.
\end{equation}
which is \eqref{eq1.7} with $f_n=f /G_n$ and by Theorem \ref{thm1.3} it follows that $\frac1n\,\mc{X}\(Q_n\)\overset{*}{\to}\lambda_F$. Then \eqref{eq1.9} together with Hermite interpolation formula imply \eqref{eq2.3} (theorem \ref{thm1.3} for this case was actually proved in \cite{St86}).
The key point in the construction of the proof above is investigation of complex orthogonal polynomials $Q_n$ which was possible to carry out using
Theorem \ref{thm1.3} since associated Green's $S$-problem without external field has always a positive solution \cite{St85} (we note that this Greens problem is equivalent to some logarithmic $S$-problem with external field generated by a negative charge.

% Sec. 6.3
\subsection{Best rational approximations to the exponential on $\mbb{R}_+$}\label{subsec6.3}
More general theorem on the best rational approximation to a sequence of analytic functions was proved in \cite{GoRa87}. The proof followed the same path as in the proof of Gonchar-Stahl theorem above which leads to the complex orthogonality relations for denominators of related Pad\'e Approximants. The formula of n-th root asymptotics is then needed for these polynomials which requires solution of an $S$-problem with a harmonic external field. 
It was the context in which $S$-problem has been for the first time explicitly introduced; hypotheses of the general theorem included an assumption that related $S$-curve for Green's potential exists in a given class. Then, for an important particular case related to  a solution of the well-known ``$\frac19$-problem''  the associated $S$-problem was costructively solved and a few important remarks were made about the problem in general.

Let $\rho_n=\min\limits_{r\in\mbb{R}_n}\max\limits_{x\in\mbb{R}_+}\left|e^{-x}-r(x)\right|$. The problem of associated rate of convergence  was later called the ``$1/9$-problem'' by E. Saff and R. Varga \cite{SaVa77}, since numerical experiments showed that $\rho_n^{1/n}\approx1/9$ for large enough $n$. Originally, the problem was introduced in the end of 1960's in connection with the numerical solution of the heat-conducting equation and it attracted common attention. It was eventually solved in \cite{GoRa87}, where one can also find a review of the preceding results.

The method described above lead the following equilibrium problem. Let $\Omega=\mbb{C}\sim\mbb{R}_+$, $\mc{F}$ is a class of curves $F$ in $\Omega$ which go from $\infty$ to $\infty$ around $\mbb{R}_+$ in such a way that $\varphi(z)=\Re z\to+\infty$ as $z\to\infty$, $z\in F$ (actually we need $\varphi\to+\infty$ fast enough, a fact which presents some technical problems); instead we may consider curves $F$ from $a-i$ to $a+i$ in $\Omega$ where $a>0$ is large enough. Now the question was if there is a Green's $S$-curve $S\in\mc{T}$ in the field $\varphi$ relative to $\Omega$. 
The positive answer has been obtained in \cite{GoRa87} by construction of $S$-curve and, susequently the rate of convergence of the best approximants has been found in terms of the parameters of its equilibrium distribution. More exactly, let $\lambda=\lambda_S-\lambda_E$ be the equilibrium charge of $(S,E)$ (here $E=\mbb{R}_+$) and $w=\(V^\lambda+\varphi\)(\xi)$, $\xi\in\supp$ be the corresponding Green's equilibrium constant. Then
$$
\lim_{n\to\infty}\rho_n^{1/n}=\rho=e^{-2w}=\frac1{9.2\dotsc}.
$$
Furthermore, near-best approximants may obtain by interpolation of $e^{-2nz}$ with density represented by $\lambda_E$; $\lambda_S$ represents (contracted) zero distribution of denominators $Q_n$.

Strong asymptotics for polynomials $Q_n$ has been later obtained by A. Aptekarev \cite{Ap03} who used steepest descent method for the Matrix Riemann--Hilbert (MRH) problem; as a corollary he proved A. Magnus' conjecture on existence and value of $\lim\limits_{n\to\infty}\rho_n/\rho^n$. Steepest descent as a general mehtod of solving certain class of MRH problems has been developed by P.Deift and X.Zhou  \cite{DeZh93} (see also  the book \cite{De99}). The method is  widely used in various classes of problems;  in particular, it is a powerfull method for proving formulas of strong asymptotics for complex ortnogonal polynomials. In a general context  the method requires existence of an $S$-curve as well as the  potential theoretic method we were discussing above.

 Some progress in the investigation of Green's  $S$-problem has been made in the paper by S. Kamvissis and author \cite{KaRa05} where the $\max$--$\min$ energy method has been outlined in a particular situation related to a NLS (we will go into some details insection \ref{subsec6.6} below). In the next section a general existence theorem for Green's $S$-problem is presented similar to theorem \ref{thm3.1} above. Actually we can carry theorem \ref{thm3.1} over to the Green's case without essential modifications. Some of the required changes in assertions and proofs are briefly discussed next.

% Sec. 6.4
\subsection{An existence theorem for Green's $S$-curves}\label{subsec6.4}
We preserve the basic hypotheses of Theorem \ref{thm3.1}. That is, we assume that the external field has finite number of singular point, the class of compacts has a finite number of fixed points (now we combine the two sets and use a single set $A$) and number of components in each compact is bounded by a common constant.

A new detail in Green's case is the presence of the boundary $\partial\Omega$ of domain $\Omega$ whose Green function defines potentials. We will assume that $\partial\Omega$ has finite number of connected components. Since the whole problem is conformal invariant it may be reduced to the case of a circular domain ($\partial\Omega$ is a union of disjoint circles.) 

Next,  we assume that there is a larger domain $\Omega_0\supset\ol{\Omega}$ such that $\varphi$ is harmonic in $\Omega_0$ except for a finite set of points (the condition may, of course, be relaxed, but we do not go here into further discussion).
Finally, conditions (iia) and (iib) cannot be both satisfied for a family of compacts in a domain with large boundary. So, we will replace (iib) with a condition on one-sided variations. This will make it possible to prove the next theorem by essentially the same way we prove Theorem \ref{thm3.1} in Section \ref{sec9} below.
% Theorem 6.1
\begin{theorem}\label{thm6.1}
Let $\Omega$ be a circular domain, $\Omega_0\supset\ol{\Omega}$. Suppose that the external field $\varphi$ and class of curves $F\subset\ol{\Omega}$ satisfy the following conditions:
\begin{enumerate}
\item[(i)]  $\varphi$ is harmonic in $\Omega_0\sim{A}$ where ${A}\subset{\Omega_0}$ is a finite set

\item[(ii)]
\begin{enumerate}
\item[(a)]  $\mc{F}$ is closed in a $\delta_{H}$-metric

\item[(b)]  $F\in\mc{F}$ implies $F^t\in\mc{F}$ for small enough $t$ for any $A$-variation if $F^t\subset\ol{\Omega}$

\item[(c)]  For some $s>0$ we have $s(F)\le{s}$ for any $F\in\mc{F}$;
\end{enumerate}

\item[(iii)]
\begin{enumerate}
\item[(a)]  There exist $F\in\mc{F}$ with\newline
$\mc{E}_{\varphi}^\Omega[F]=\inf\limits_{\mu\in\mc{M}(F)}E_{\varphi}(\mu)>-\infty$;

\item[(b)]  $\mc{E}_{\varphi}^\Omega\{F\}=\sup\limits_{F\in\mc{F}}E_{\varphi}[F]<+\infty$.
\end{enumerate}

\item[(iv)]  For any sequence $F_n\in\mc{F}$ that converges in $\delta_H$ metric \tu{(}$\delta_H(F_n,F)\to 0$\tu{)} there exists a disc $D$ such that $F_n\sim D\in\mc{F}$
\end{enumerate}

Then there exists a compact $S\subset\mc{F}$ with the $S$-property.
\end{theorem}

An analogue of representation (3.5) is valid for the potential of Green's equilibrium measure $\lambda$ of $S$ (external field has to be modified; in case when $\Omega $ is upper half-plane one has to add potential of a negative "mirror reflection" of $\lambda$ to the external field).

Proof of the theorem follows all the steps of $\max$--$\min$ energy method, which are presented in Section \ref{sec9} for logarithmic potential. First, we need to study continuity properties of the equilibrium energy functional
\begin{equation}
\label{eq6.5}
\mc{E}_\varphi^\Omega[F]=\inf_{\mu\in\mc{F}(F)}\mc{E}_\varphi^\Omega(\mu):\mc{F}\to[-\infty,+\infty]
\end{equation}
on $\mc{F}$ with the Hausdorff metric and establish

%\textbf{Theorem 6.2}
\begin{theorem}\label{thm6.2}
The functional $\mc{E}_\varphi^\Omega$ above is upper semi-continuous.
\end{theorem}

Like the proof of theorem \ref{thm3.2} in Section \ref{sec3}, it is convenient to introduce a larger class $\mc{K}$ of compacts $K$ in $\Omega$ satisfying condition $s(K)\le s$, for any $K\subset\mc{K}$. For continuous external field $\varphi$ the functional
$$
\mc{E}_\varphi^\Omega:\mc{K}\to(-\infty,+\infty)
$$
is continuous. Then fields with singular points are approximated by continuous fields. Theorem \ref{thm6.2} implies that there exist the extremal compact $s$ with
\begin{equation}
\label{eq6.6}
\mc{E}_\varphi^\Omega[S]=\sup_{F\subset\mc{F}}\mc{E}_\varphi^\Omega[F];\quad S\in\mc{F}.
\end{equation}
which leads to the following theorem.

%\textbf{Theorem 6.3}
\begin{theorem}\label{thm6.3}
The compact $S$ in \eqref{eq6.6} above has Green's $S$-property.
\end{theorem}

As in case of logarithmic potebtial the proof of theorem \ref{thm6.3} is based on the fact that equilibrium measure of extremal compact $S$ in \eqref{eq3.6} is a Green's critical measure associated with external field $\varphi$ and fixed set $A$. The part of the proof related to variations may be completely reduced  to the case of logarithmic potential. Some details related specifically to Green's potential are discussed in \cite{KaRa05}.

% Sec,. 6.5
\subsection{Best rational approximations to the sum of exponentials on $\mbb{R}_+$}\label{subsec6.5}
We briefly mention an approximational problem which presents a typical situation when conditions of theorem \ref{thm6.1} are not satisfied. Let
$f(x)=\sum\limits^m_{k=1}c_k e^{-\lambda_k x}$ where $\lambda_k>0$ and $C_k\in\ol{\mbb C}$; consider $\rho_n=\min\limits_{\Gamma\in\mbb R}\max\limits_{x\in\mbb R_+}|f(x)-r(x)|$. We want to find if $\lim\limits_{n\to\infty}\rho_n^{1/n}$ (if it exists) and to construct a near best approximation. The problem is clearly a direct generalization of the case $m=1$ discussed in Section \ref{subsec6.3} above. The interpolation approach in Section \ref{subsec6.3} is applied and leads (through orthogonal polynomials as in Section \ref{subsec6.4}) to a similar $S$-problem for Green's potential in the domain $\Omega=\mbb{C}\sim\mbb{R}_+$ for class $\mc{F}$ of curves $F\subset\Omega$ connecting $a-i$ and $a+i$ with a large $a>0$. The difference with the case $m=1$ is that external field $\varphi=\Re z$ has to be replaced with the field
$$
\varphi(z)=
\begin{cases}
\ \Re z, \ \ \Re z >0\\
\lambda\Re z, \ \Re z <0\\
\end{cases}
%\bigg\}
\quad\lambda=\frac{\min\lambda_k}{\max\lambda_k}
$$
This function is not harmonic in $\Omega$ (its singularities constitute a whole line --- imaginary axis $I$) and theorem \ref{thm6.1}  does not apply. One can verify that the $\max$--$\min$ method allows us to prove existence of extremal compact $F_0\in\mc{F}$ which satisfies $\mc{E}^\Omega_\varphi\[F_0\]=\max\limits_{F\in\mc{G}}\mc{E}_\varphi^\Omega[F]$. 
If $\cop\(F_0\cup I\)=0$ then $F_0$ has the $S$-property. Most likely that $F_0\cup I$ consists of two points, but the proof of this fact is not known. The problem is open.

% Sec. 6.6
\subsection{Non-Linear Schr\"{o}dinger}\label{subsec6.6}
Another example of a Green's $S$-problem which does not satisfy conditions of theorem \ref{thm6.1} come from problem of semiclassical limit for the focusing $NLS$. The problem may be reduced to the following $S$-problem for Green's potential. Let
$$
\varphi(z)=Re(az^2+bz+c)+V^\sigma(z)
$$
where $\sigma$  is a positive measure on the interval $\Delta= [0, iA]$ of imaginary axis $d\sigma(iy)=\eta(y)dy$, $y\in[0,A]$ and $\eta$ is an analytic  ($\eta(y)=1$ would be a typical example). Coefficients $a,b,c$ depend on original variables $x,t$ so that we have a family of quadratic functions; see \cite{KaMLMi03, KaRa05} (see also \cite{MiKa98, ToVeZh04,ToVeZh07, BeTo10} for other reductions).

Let $\Omega$ be the upper half-plane and $\mc{F}$ be a set of all curves $F$ which begins at zero on one side of $\partial(\Omega\sim\Delta)$ and go in $\Omega$ around $\Delta$ to the point zero on the opposite side of $\partial(\Omega\sim\Delta)$.

Now, the problem is to determine if there is a compact $S\in\mc{F}$ with Green's $S$-property relative to $\Omega$ and $\varphi$. We note that the problem is very similar to the $S$-problem related to rational approximations for sum of exponentials in Section \ref{subsec6.4} above. The resemblance may apparently be explained by the fact that both problems may be technically formulated using a matrix Riemann--Hilbert problem with $2\times2$ matrices of a similar structure.

The two problems in sec \ref{subsec6.5} and \ref{subsec6.6} have the same property - external field $\varphi$ has a 
line singularity. We may still consider $\max$--$\min$ energy problem and find $F_0$ with
$$
\mc{E}_\varphi\[F_0\]=\sup_{F\in\mc{F}}\mc{E}_\varphi[F]
$$
It is not difficult to prove that $F_0$ (which belongs in general to $\ol{\Omega}$) does not have a large intersection with $\mbb{R}$. (Actually $F_0\cap\mbb{R}$ consists of one or two copies of zero.) However an intersection of $F_0$ and $\Delta$ may have a positive capacity and in such cases we actually do not know if original $S$-problem has a solution or not.

Some analysis of the case is presented in \cite{KaRa05}, but the problem remains open.

% Se. 6.7
\subsection{$S$-problem for vector potentials}\label{subsec6.7}
For an integer $p\ge 1$ consider $p$ classes of compacts $\mc{F}_j\subset\mc{K}$ with an external field $\varphi_j$ s defined on $F_j\in\mc{F}_j$,  $j=1,\dotsc,p$. For a fixed vector-compact $\vec{F}=\(F_1,\dotsc,F_p\)\in\mc{F}$  and a vector $\vec {m} =  (m_1,\dots, m_p)$ with positive components
(total masses on componrnts of vector-compact) we define a family of vector-measures
$$
\vec{\mc{M}}=\vec{\mc{M}}\(\vec{F},\vec{m}\)=\left\{\vec{\mu}
=\(\mu_1,\dotsc,\mu_p\):\mu_j/m_j\in\mc{M}\(F_j\)\right\}.
$$
For a fixed positive (or non-negative) definite real symmetric matrix $A=\left\|a_{ij}\right\|^p_{i,j=1}$ we define  associated  energy and, susequently, weighted energy on $\vec{\mc{M}}$
$$
\mc{E}\(\vec{\mu}\)
=\sum^P_{i,j=1}a_{ij}\[\mu_i,\mu_j\];\quad
\mc{E}_\varphi\(\vec{\mu}\)
=\mc{E}\(\vec{\mu}\)+2\sum^p_{j=1}\int\varphi_j d\mu_j
$$
where $[\mu,\nu]=\int V^\nu d_\mu$-mutual energy of $\mu$ and $\nu$.
\begin{lemma}
\label{lem6.4}
If $A$ is non-negative definite and $a_{ij}\ge 0$ for $F_i\cap F_j\ne\emptyset$ then there exist a unique $\vec{\lambda}\in\vec{\mc{M}}$ (equilibrium measure for $\(\vec{\varphi},\vec{F},\vec{m}\)$)
$$ \mc{E}_{\vec{\varphi}}(\vec{\lambda})=\min_{\vec{\mu}\in\vec{\mc{M}}}\mc{E}_\varphi(\vec{\mu}). $$
(If $A$ is positive definite we may admit small intersections of such components). 
%$\cop\(F_i\cap F_j\)=0$ for $a_{ij}<0$
\end{lemma}
For details see original papers \cite{GoRa81, GoRa85,GoRaSo97}).
Let $\varphi_j$ be harmonic in a neighborhood of $F_j$. We say that $\vec{S}\in\vec{\mc{F}}$ has $S$-property if components of associated vector-equilibrium measure $\lambda_j$ satisfy the follwoing conditions; there exist a set $e$ of zero capacity such that for any $\zeta\in\supp(\lambda_j)\sim e$ there exist a neighborhood $D=D(\zeta)$ for which $\supp(\lambda_j)\cap D$ is an analytic arc and, furthermore, we have
$$
\frac{\partial W_j}{\partial n_1}(\zeta)
=\frac{\partial W_j}{\partial n_2}(\zeta),\quad
\zeta\in\(\supp\lambda_j\)\sim e;\quad
W_j=\sum^P_{i=1} a_{i,j}V^{\lambda_j}+\varphi_j
$$
where $n_1,n$ are opposite normals to $\supp(\lambda_j)$ at  $\zeta$.

The vector $S$-existence problem is more complex than the scalar one. We mention briefly some details connected with the vector $\max$--$\min$ method.  For fixed $ A,\ \vec{m}$ we define the minimal energy functional $\mc{E}_{\vec{\varphi}}[\vec{F}]$  on $\vec{\mc{F}}$  and, then, corresponding extremal vector-compact $\vec{S}$
$$
\mc{E}_{\vec{\varphi}}[\vec{F}]\ =\inf_{\vec{\mu}\in\vec{\mc{M}}(\vec{F})}\mc{E}_\varphi(\vec{\mu})=\mc{E}_\varphi(\vec{\lambda});\quad\quad
\mc{E}_\varphi[\vec{S}]=\sup_{\vec{F}\in\mc{F}}\mc{E}_\varphi[\vec{F}].
$$

%Suppose there exist $\vec{S}\in\vec{\mc{F}}$ which maximizes $\mc{E}_{\vec{\varphi}}$
%$$ \mc{E}_\varphi[\vec{S}]=\sup_{\vec{F}\in\mc{F}}\mc{E}_\varphi[\vec{F}].$$

Suppose that small variations of $\vec{S}$ with a vector set $\vec{A}$ of fixed points belongs to $\vec{\mc{F}}$ and also that $S_i\cap S_j$ is a finite set for any two components of $S$. Then $S$ has the $S$-property. This fact is rather simple corollary of variational technique described in section \ref{sec4}. The existence of maximizing $\vec{S}\in\vec{\mc F}$ may also be established under essentially same general conditions as in scalar case  (e.g., a finite set of singularities of $\varphi_j$, a finite fixed set $A_j$, a finite number of component of $F_j$; see section \ref{sec9}). But the intersections $S_i\cap S_j$ is more difficult to control.

At the same time there is comparetively general approach to the vector $S$-problem from the theory of Riemann surfaces; see first of all J.Nuttall's papers \cite{Nu81,Nu84}. In addition, a large number of particular situations are investigated in connection with the applications of vector equilibrium problems to approximation theory, random matrices, statistics, etc.; see  \cite{GoRaSu92, BlKu05, ApKuVa08, KuMaWi09} and references therein.

%  Section 7
\section{Strong asymptotics}\label{sec7}

Preceeding sections have been devoted to weak asymptotics -- zero distribution of orthogonal polynomials $Q_n$ with varying weights. In case of fixed (not depending on $n$) weight it may be characterized by an equilibrium problem without external field. Formulas of strong asymptotics for such polynomials may also be stated in electrostatic terms; then associated equilibrium problems contain an external field of order $O(\frac1n)$. It is convenient, for a fixed $n$ to renormalize the problem and consider measures of total mass $n$. Then the external field does not depend on $n$ at least in "simply-connected' situations like one in the next section \ref{subsec7.1}. In general, there is a bounded "effective external field" which depend on $n$ and has several components. In this section we go into some detail in case of orthogonality on one or several intervals of real axis.
% 7.1
\subsection{Bernstain-Szeg\H{o} formulas in electrostatic terms}\label{subsec7.1}
We consider orthogonal polynomials $Q_n(x) = x^n + \dots$ on the interval $\Gamma = [-1, 1]$ 
$$ \int_\Gamma Q_n(x)\ x^k~w(x)\,dx\ =\ 0,
\quad k=0,1,\dotsc,n-1.$$ 
with a smooth weight $w$ which is positive except, may be, for a finite set of points.  If the Szeg\H{o} condition $\ \log w \in L (\Gamma)\ $ is satisfied  we have
\begin{equation}
\label{eq7.1}
Q_n(z)=W_n(z)\(1+\epsilon_n(z)\),\quad
\epsilon_n(z)\to 0,\quad z\in\ol{\mbb{C}}\sim\Gamma
\end{equation}
\begin{equation}
\label{eq7.2}
Q_n(x)=W_n^{+}(x) + W_n^{-}(x) + \wt{\epsilon}_n(x),\quad
 2^n\wt \epsilon_n(x)\to 0,\ x\in\Gamma
\end{equation}
% 7.3
\begin{equation}
\label{eq7.3}
W_n(z)={C_n} D(z){\(z^2-1\)^{-\frac14}}\(z+\sqrt{z^2-1}\,\)^{n+\frac12}
\end{equation}
where the Szeg\H{o} function $D(z)\in H(\ol{\mbb{C}}\sim\Gamma)$ is determined by boundary values
 %(trigonometric Szeg\H{o} function) - $D_0(z)\in H(\ol{\mbb{C}}\sim\Gamma)$ W_n(z)=C_n D_0(z)\(z+\sqrt{z^2-1}\,\)^n
% 7.4
\begin{equation}
\label{eq7.4}
\left|D^\pm(x)\right|^{-2}=w(x), \quad(D(\infty)>0),  \quad x\in\Gamma; 
\end{equation}
($C_n=1/2^{n+\frac12}D(\infty)$ is normalization constant; in what follows we do not write such constant explicitly). We note that \eqref{eq7.2} is valid apart from zeros of $w$ and endpoints of $\Gamma$; 
see \cite{Sz75, Ne86, Si05a, Si05b} for details. 
%\eqref{eq7.3} is written asFunctions  $w_0$ is called trigonometric weight. Sometimes it is more convenient to use 
%|D^\pm_0(x)|^{-2}=w_0(x)=w(x)\sqrt{1-x^2},\quad x\in\Gamma \quad(D(\infty)>0)
%\varphi_0(x)=\frac12\log\frac1{w_0(x)}

Next, we introduce an external field $\varphi$ associated with the weight $w$
% 7.5
\begin{equation}
\label{eq7.5}
\varphi(x)=\frac12\log\frac1{w(x)}, \quad {\text or}, \quad w(x)=e^{-2\varphi(x)} 
\end{equation} 
 %D_0(z)=\frac{\(z+\sqrt{z^2-1}\,\)^\frac12}{\(z^2-1\)^{1/4}}D(z);\quad
%$\mc{M}_n(\Gamma)$ e^{-2\varphi_0(x)}$.
%\begin{gather*}
%\varphi_0(x)=\frac12\log\frac1{w_0(x)};\quad
%\varphi(x)=\frac12\log\frac1{w(x)} \\
%(\text{that is }w(x)=e^{-2\varphi(x)}
%=\frac1{\sqrt{1-x^2}}\,e^{-2\varphi_0(x)}).
%\end{gather*}
and express function $W_n$ in  \eqref{eq7.3} above  in terms of equilibrium measures in the field $\varphi$. Note that for a moment we are concerned not with the error of asymptotics $Q_n\approx W_n$ and not with sharp conditions on weight but rather with the form of $W_n$. Let $\mc M_n(\Gamma)$ be the set of positive Borel measures on $\Gamma$ with total mass of $n$ and $ \wt\lambda\in \mc M_{n+\frac12}$ be the equilibrium measure in the field $\varphi$. We define
$$\lambda_n= \wt{\lambda}-\frac14\,\delta(1)-\frac14\,\delta(-1) \ \in \mc M_{n}$$ 
where $\delta$ is the Dirac $\delta$-function. We have (exact or approximate) representations for $W_n$ with $z\in\mbb C \sim\Gamma$ and  $x\in\Gamma$ in terms of $\lambda_n$
%  be the the equilibrium measure with total mass $n+\frac12$. In the most simple case when
%$\varphi_0\in C'(\Gamma)$, that is $w_0\in C^1\(\Gamma\)$
%$w_0>0$ on $\Gamma$ the external field $\varphi_0$ is smooth. 
%Then we have $\supp\lambda=\Gamma$ for large enough $n$ and $V^\lambda+\varphi_0=$ const on $\Gamma$. It follows that
% 7.6
\begin{equation}
\label{eq7.6}
W_n(z)=e^{-\mc{V}^{\lambda_n}(z)},\quad\quad
W_n^{+}(x) + W_n^{-}(x) = A_n (x) \cos\Phi_n(x)
\end{equation}
%\lambda=\lambda_{\varphi_0}\in \mc M_n
where $\mc{V}^{\lambda_n}={V}^{\lambda_n}+i\ \wt{{V}}^{\lambda_n}$ is the complex potential of $\lambda_n$ and
% 7.7
\begin{equation}
\label{eq7.7}
 A_n (x) = C_n {(1-x^2)^{-\frac14} {w^{-\frac12}(x)}}, \quad  \Phi_n (x) =\pi\lambda_n\{[x,1]\} =\pi \int_x^1 d\lambda_n
\end{equation}

We have exact equality in \eqref{eq7.6} if $w>0$ and $w\in C^1(\Gamma)$ so that $\varphi\in C^1(\Gamma)$ and $n$  is large enough. Otherwise this representation is approximate which is not significant for applications of $W_n$ in asymptotic formulas since the error of approximation here is generally not large compare to $\epsilon$ in \eqref{eq7.1}. Measure $\lambda_n$ is not positive which is not important too since it can be approximated by a positive measure and this can made in many way. One way is to cut two parts  of magnitude $\frac14$ each from the measure $\wt \lambda$ near the two endpoints of the support. 

Another way is to define  $\lambda_n$  the equilibium measure with normalization $\mc{M}_n(\Gamma)$ in the external field $$\varphi_0(x)=\frac12\log(1/w_0(x))  \quad {\text where} \quad  w_0(x)=(1-x^2)^{\frac12}w(x)$$
($w_0$ is the trigonometric weight) in these terms we have $A_n(x)= C_n {w_0}^{-1/2}(x) $ in place of \eqref{eq7.7}. A generalization of this way to introduce $\lambda_n$ is actually used in multiconnected cases in sections \ref{subsec7.2},  \ref{subsec7.3} and section  \ref{sec8}  below.

Anyway, using electrostatic terms one can define a measure $\lambda_n \in \mc M_n(\Gamma)$ which plays a role of "a fine, continuous model"  for zero distribution of orthogonal polynomial $Q_n$. More exactly, zeros of $Q_n$ are distributed uniformly with respect to $\lambda_n$ (apart from zeros of $w$ and endpoints of the $\Gamma$) with errors small with respest to distances between zeros. Equivalently, $\lambda_n$-measure of an interval between two susequent zeros of $Q_n$ is equal to $1+o(1)$. Exponential of complex potential of $\lambda_n$ may be, then, viewed as a continuous model  for the orthogonal polynomial itself (see \eqref{eq7.7}). In multiconnected cases below we will use generalizations of this definition.

%and for equilibrium measure $\wt{\lambda}=\lambda^{n+1/2}_\varphi$ we have 
%$V^{\wt{\lambda}}+\varphi=\const$ on $\Gamma$ and, therefore,

%$W_n(z)=\(z^2-1\)^{-1/4}e^{-\mc{V}^{\wt{\lambda}}(z)},
%\quad\wt{\lambda}=\lambda^{n+1/2}_\varphi$
%If neither of the functions $\varphi$, $\varphi_0$ is smooth, then neither of the equalities \eqref{eq7.6} or \eqref{eq7.7} is exact, but both are approximate with an extra multiplier $(1+\epsilon_n(z))$ in the right-hand side with $\mc{E}_n\to 0$ in $\ol{\mbb{C}}\sim\Gamma$. So that both representation for $W_n$ may be used in asymptotic formula \eqref{eq7.2}. (Note that the difference between $\lambda$ and $\wt{\lambda}$ is easily accountable, as we have $\wt{\lambda}\approx\lambda+\frac12\,\delta(1)+\frac12\,\delta(-1)$ where $\delta$ is the Dirac $\delta$-function.)

Representations in terms of $\lambda_n$ are, in particular, useful when $\varphi(x)$ is smooth but it has, say, a finite number of exponential zeros such that the Szeg\H{o} condition
$$
\(1-x^2\)^{-1/2}\log w(x)\in L_1(\Gamma).
$$
is not satisfied and the Szeg\H{o} function  $D$ does not exist. Subsequently, $W_n$ in \eqref{eq7.3} does not exist, however, the approximation of this function in \eqref{eq7.6}  exists. Asymptotic representation  \eqref{eq7.1}  (and  also \eqref{eq7.2} away from zeros of $w$)  is valid with $W_n$ defined in terms of equilibrium measures in  \eqref{eq7.6}.  It seems that  this fact was not generally proved yet; see some discussion in \cite{LoRa88, LeLu01}.

Significant progress has been obtained in the last three decades by the application of potential-theoretic methods in the investigation of polynomials orthogonal on the whole real axis and more general noncompact sets. Discussion of these results goes beyond the scope of this pape; see, in particular, original papers \cite{Ra82, GoRa84, MhSa84, Ra93} and books \cite{StTo92, LeLu01, SaTo99}
% Sec. 7.2
\subsection{Orthogonality on several intervals}\label{subsec7.2} Situation become more complicated when from the case of a single interval $\Gamma$ 
 we pass to the case when $\Gamma $ is a union of $p>1$ dsjoint intervals $\ \Gamma_k=\[a_{2k-1},a_{2k}\]$, $\ k=1,\dotsc,\ p\ $ where $a_1<a_2<\dotsb<a_{2p}$ are real.

First results on strong asymptotics for polynomials $Q_n$ orthogonal on $\Gamma$ with a (smooth enough, positive) weight $w(x)$ were obtained by N. Akhiezer \cite{Ak60}. He observed, in particular, that floating zeros of polynomials (zeros located in gaps between intervals) are determined by the Jacobi Inversion Problem (JIP). Systematic study of the case were included in the well-known monograph by H. Widom \cite{Wi69}. 

The asymptotic formula for $Q_n$ is constructed from special functions associated with the domain $\Omega=\ol{\mbb{C}}\sim\Gamma$. Let $\ \omega_k(z)$ be the harmonic measure of $\Gamma_k$; that is,  $\omega_k$  is harmonic in $\Omega$ and $\ \omega_k=1\ $ on $\Gamma_k\ $, $\ \omega_k=0\ $ on $\Gamma_j$, $j\ne k$;  $\ \Omega_k(z)=\omega_k(z)+i\ \wt{\omega}_k(z)$ is corresponding analytic functions . For $\zeta\in\Omega\ $ let $\ g(z,\zeta)\ $ be the Green function with logarithmic pole at $\zeta$ and $\ G(z,\zeta)=g(z,\zeta)+i\ \wt{g}(z,\zeta)\ $ be corresponding complex Green function.
Finally, we assume that the function $\varphi(x)=-\frac12\log w(x)$ has a harmonic continuation to $\Omega$ (Szeg\H{o} condition). So, there exist a Szeg\H{o} function (with multi-valued argument)
$$
D(z)=\exp (\varphi(z)+i \wt{\varphi}(z));
\quad |D^\pm(x)|^{-2}=w(x).
$$

Asymptotic formula includes also $p-1$  couples of parameters
$$
\zeta_{k,n}\in\wt{\Gamma}_k
=\[a_{2k},a_{2k+1}\],\quad 
s_{k,n}=\pm 1\quad k=1,\dotsc,p-1
$$
For each $n$ the collection $\left\{\zeta_{k,n},s_{k,n}\right\}$ is uniquely defined by a JIP  \eqref{eq7.10} below. Points $\zeta_{k,n}\in\Omega$ with $s_{k,n}=1 $ asymptotically represent zeros of $Q_n$ in the gaps between intervals in $\Gamma$. We define
$$ h_n(z)=\prod_{k=1}^{p-1}\(z-\zeta_{k,n}\)\  , \quad\quad
A(z)=\prod_{k=1}^{2p}\(z-a_{k}\). $$

With this notations Akhiezer--Widom asymptotic formula (see theorem 6.2  in \cite{Wi69}: see also sec.14 there)  asserts that for $z\in\Omega, \ x\in\Gamma$ we have 
% 7.8
\begin{equation}
\label{eq7.8}
Q_n(z)=W_n(z)\(1+\epsilon_n(z)\),\quad Q_n(x)=W_n^{+}(x) + W_n^{-}(x) + \wt \epsilon_n(x)
\end{equation}
where $\epsilon_n(z)\to 0$ uniformly on compacts in $\Omega$ except for small neighborhoods of points $\zeta_{k,n}$ ($ \wt \epsilon_n $ is small   with respect to $|W_n|$ in $L^2$-norm) and $W_n(z)$ is defined by
\begin{equation}
% 7.9
\label{eq7.9}
W_n(z)=C_nD(z)\frac{\sqrt{h_n(z)}}{A^{1/4}(z)}\exp\left\{\(n+\frac12\)
G(z,\infty)-\frac12\sum^{p-1}_{k=1}s_kG(z,\zeta_{k,n})\right\}
\end{equation}
We note that $W_n\(\zeta_{k,n}\)=0$ if $s_{k,n}= 1$ and $W_n\(\zeta_{k,n}\)\ne 0$ if $s_{k,n}= -1$.

Parameters $\zeta_{k,n}$, $s_{k,n}$ are defined by the following  system of equations 
% 7.10
\begin{equation}
\label{eq7.10}
\sum_{k=1}^{p-1} s_{k,n}\omega_j\(\zeta_{k,n}\)
=\gamma_{j }  +n\omega_j(\infty) \ \ [\mod 2 \ ],\quad j=1,\dotsc,p-1
\end{equation}
where $\pi\gamma_{j}$ is the period of $\arg D$ around $\Gamma_j$. System of equations \eqref{eq7.10} is a standard Jacobi Inversion Problem which represent the condition that $W_n(z)$ in \ref{eq7.9} is single-valued in $\Omega$. 

A version of the formula for the case of varying analytic weight is contained in \cite{DeKrMLVeZh99-2}; a general stucture of the function $W_n$ is preserved in case of varying weght.

% 7.3
\subsection{Electrostatic Interpretation of the Akhiezer--Widom formula}\label{subsec7.3}

Here we interprete the function $W_n(z)$ in  \eqref{eq7.9}, \eqref{eq7.10} in electrostatic terms as an exponential of a complex equilibrium potential. The construction of corresponding measure $\mu_n$ is similar to the one in section \ref{subsec7.1} but together with the main (continuous)  $\lambda$-component (we  use   "trigonometric" version of construction for this component)  it contains a finite discrete $\nu$-component reflecting floating zeros of $Q_n$ and. In addition, the formula contains an extra set of parameters  $\zeta_{k,n}\in\[a_{2k},a_{2k+1}\]$ and $s_{k,n}=\pm 1$, $k=1,2,\dotsc,p-1$.

It is, first of all, convenient to represent set of parameters by a  single charge $\sigma_n$; the we define charge $\nu_n$ which wirepresent floating zeros of $Q_n$. Thus, we define
\begin{equation}
\label{eq7.11}
\sigma_n=\frac12\ \sum_{k=1}^{p-1}s_{k,n}\delta\(\zeta_{k,n}\)\quad\quad
\nu_n=\ \sum_{s_{k,n}=1}\delta\(\zeta_{k,n}\)
\end{equation}
where $\sigma(\zeta)$ is unit mass at $\zeta$ (in other terms $2\sigma_n$ and $\nu_n$ are divisors on the Double of $\Omega$).
Next, for each $n\in\mbb N$ we define an external field $\varphi_n(x)$ on $\Gamma$
\begin{equation}
\label{eq7.12}
\varphi_n(x)=\frac12\log\frac{1}{w(x)} - \frac14 \log|A(x)| +  V^{\sigma_n}(x)
 \end{equation} 
We note that $\ \exp \{ -2 \varphi_n\} $ presents one of the versions of the "trigonometric weight" for multiconnected case.
 
Finally, let $l_n=\frac12\sum\limits_{k=1}^{p-1}\(s_{k,n}+1\)$ be the number of positive $s_{k,n}$. Then, let $\lambda_n$ be the equilibrium measure on $\Gamma$  of total mass $n-l_n$ in the field $\varphi_n$ a and $\ \mu_n = \lambda_n +\nu_n$. 

% Theorem 7.1
\begin{theorem}\label{thm7.1}
Suppose that $w_0(x) = w(x)\sqrt{|A(x)|}$ is positive and smooth. Then for each large enough $n\in\mbb{N}$ there exist unique $\sigma_n$ of the form \eqref{eq7.11} such that corresponding equilibrium measure $\lambda_n$ satisfies condition that 
$$\lambda_n\(\Gamma_j\) \ \in \ \mbb Z \ , \quad j=1,\dotsc,p-1 .$$ With this $\lambda_n $ we have 
$$
W_n(z) =   e^{-\mc{V}^{\mu_n}(z)}  =  e^{-\mc{V}^{\lambda_n}(z)}  \prod_{s_{k,n}=1}\(z-\zeta_{k,n}\)  .  
$$ 
\end{theorem}

The proof is rather streightforward and eventually reduces to comparison of boundary values of $\log\left|W_n(x)\right|$ and of the potential of $\mu_n$. Conditions on $w_0$ may be essentially relaxed; then exact representation for $W_n$ in the theorem \ref{thm7.1} above become approximate. 

%In short, we know in advance that there exists representation \eqref{eq7.12} with some measure $\lambda_n$ with $\lambda_n(\Gamma)=n-l_n$. Then, conditions that  $w_0\in C^1(\Gamma)$ and is positive impliy that $\lambda_n$ is a positive measure for large enough $n$ and we also have $\supp\lambda_n=\Gamma$. Now \eqref{eq7.12} follows from

Thus, in the real multi-connected case measure $\mu_n = \lambda_n+\nu_n$ representing in a strong sense zeros of orthogonal polynomial is still  defined in terms of equilibrium measure with an external field depending on  weight function. Compare to simpliconnected case the external field is modified by adding a potential of a discrete charge \eqref{eq7.11}. The charge is selected to make total mass of $\lambda_n$ an integer on each component of the support of weight. It turns out that essentially the same is true for complex orthogonal polynomials in more general situations. A particular case is presented in the next section.

% Section 8
\section{Generalized Jacobi polynomials}\label{sec8}
A theorem has been mentioned in section \ref{sec2}  on zero distribution of Pad\'e denominators $Q_n(z)$ associated with an element $f\in\mc{A}\ (\ol{\mbb{C}}\sim A)$ where $A=\left\{a_1,\dotsc,a_p\right\}$ is a finite set (set of branch points of $f$). Here we make some remarks on the problem of strong asymptotics for an important particular case of an elements  of the form \eqref{eq8.1} below; the case was already briefly discussed in section \ref{subsec4.2}.

Let $S$ be Stahl's continuum for $f$, that is $f\in H\(\ol{\mbb{C}}\sim S\)$ and $S$ have minimal capacity in this class.  Polynomials $Q_n(z)=z^{n'}+\dotsb$ ($n'\le n$) -- Pade denomiinators for $f$ satisfy complex orthogonality relations
$$
\oint_S Q_u(z)z^k f(z)\,dz
=\int_S Q_n(z)z^kw(z)\,dz=0,\quad k=0,1,\dotsc,n-1
$$
where $w(z)=f^+(z)-f^-(z)$ on $S$ (if $w\in L_1(S)$, otherwise some of the points $a_k$ has to be bypassed by small loops in $\Omega\subset\ol{\mbb C}\sim S$).

An important fact is that any interval or collection of intervals of real axis has $S$-property with respect to $\varphi=0$ (or any real-symmetric $\varphi$). Thus, the real case $A\subset\mbb{R}$  discussed in  \ref{subsec7.2} presents an example of our current situation. More exactly, the real case is rather good model of so-called hyperelliptic situation when Stahl's continuum $S$ is a union of disjoint Jordan arcs. 

For hyperelliptic case with additional condition that  $f(z)$ has  only quadratic branch points J. Nuttall and R. Singh  \cite{NuSi77}  proved an asymptotic formula similar to the Akhiezer--Widom formula  \eqref{eq7.8} -- \eqref{eq7.9}. We have to add that the Szeg\H{o} function $D(z)=D(z,w)$  may be defined in complex case as a solution of the boundary value problem $D^+(\zeta)D^-(\zeta)=1/w(\zeta)$, $\zeta\in S$, $D(\infty)\ne 0$ (also equation \eqref{eq7.10} has to be wrutten in terms of first kind integrals). It was later found out that the method in \cite{NuSi77}  is similar to the method in an earlier work by N. Akhiezer.  Nuttall also stated a general conjecture on the strong asymptotics formula for Pad\'e denominators in terms of the solution of a boundary value problem for couple of functions analytic in $\ol{\mbb C}\sim S$ (somewhat similar boundary value problem is known for hermitian orthogonal polynomials; see \cite{Wi69}). For details see \cite{Nu84}.
%See also subsequent papers  \cite{Su } 

The proof of a strong asymptotic formula for $Q_n$  was not known until recently for hyperelliptic case without the assumption that branch points are quadratic. General case was essentially completely open.

Comparetively general results on strong asymptotics of Pade denominators were recently obtained independently in \cite{ApYa11} and in \cite{MaRaSu11}. The MRH is used in the first of the two papers and both method MRH and Liouville-Green (WBK) are
used in the second one (with the purposes to compare the methods in a sutuation when both of them may be applied).
Here we discuss in some detail results from \cite{MaRaSu11} on strong asymptotics for generalized Jacobi polynomials $Q_n$ --- Pad\'e denominators associated with the function
\begin{equation}
\label{eq8.1}
f(z)=\prod_{k=1}^{p}\(z-a_k\)^{\alpha_k},\quad\alpha_k\in\mbb{R},\quad\sum\alpha_k=0.
\end{equation}
For $p=2$ this reduces to classical Jacobi polynomials (actually --- a particular case of them, because of condition $\alpha_1+\alpha_2=0$). The case of $p=3$ and arbitrary point $a_k$ have been investigated by J. Nuttall in \cite{Nu86}; he used Laguerre differential equations for $Q_n$ and, then, LG- asymptotics for solutions of this equation. In \cite{MaRaSu11} we combined methods from  \cite{Nu86},  \cite{MaRa10} and  \cite{MaRa11}.

Let  $S$ be the Stahl's compact for $f$ and  $V(z)=z^{p-2}+\dotsb$ be associated polynomial, so that the complex Green function for $S$ has representation $G(z)=\int^z_{a_1}\sqrt{V(t)/A(t)}\,dt$. To simplify the discussion we will assume that $S$ is a continuum, that is, all zeros $\ v_1, \dots, v_{p-2} \ $ of $\ V(z)\ $ are distinct (this is in a sense a generic case); so, Stahl's compact is now Chebotarev's continuum.  We also assume that each $a_j\in A$ is a branch point of $f$, that is $\alpha_j$ is not an integer; then $A$ and $V$ do not have common zeros and genus of the Riemann surface $\mc{R}$  associated with $\sqrt{AV}\ $ is $\ p-2$.
%$V(z)=\prod\limits_{j=1}^{p-1}\(z-v_j\)$

Next, we have to introduce special functions for $\mc{R}\ $. They are similar to (but not identical with) special functions of the domain $\ \ol{\mbb C}\sim [-1,1]$  used in section \ref{sec7} to construct Akhiezer-Widom formula. It is convenient to introduce them for now as analytic functions in the domain $\Omega=\mbb C\sim S$. We define for $z\in\Omega$ 
% we have to work with functions on Riemann surface $\mc{R}$ as covering over the plane and we need generally to distinguish points on $\mc{R}$ and their projections on $\mbb C$
$$
\Omega_k(z)=\int^z_{a_1}\sqrt{{V(t)}/{A(t)}}\ 
\frac{dt}{t-v_{k}},\quad k=1,\dots,p-2
$$
-- basis of first-kind integrals (analogues of harmonic measures),
$$
G(z,\zeta)=\sqrt{{A(\zeta)}/{V(\zeta)}}\ 
\int^z_{a_1}\sqrt{{V(t)}/{A(t)}}\ \frac{dt}{t-\zeta}
$$
-- third-kind integrals --  Green functions for $\mc R$ (here $\zeta$ is finite,  $G(z)=G(z,\infty)$ has been defined above); these functions  are multivalued in $\Omega$ but their real parts are single-valued.. All the functions above have analytic continuation to $\mc R$ which is in general  multivalued.
 
Like in the real case in section \ref{sec7}  asymptotic formula for $Q_n$ include a set of parameters $\ \zeta_{k,n}\in\Omega,\  s_{k,n}=\pm 1\  $.
With $\ h_n(z)=\prod\limits_{k=1}^{p-2}\(z-\zeta_{k,n}\)\ $ we have (see \cite{MaRaSu11}).  
%we will have parameters $\{ \wt{\zeta}_{k,n}\in\mc{R}\}$ whose projections on the plane are still denoted by $\zeta_{k,n}$.

%Theorem 8.1
\begin{theorem}\label{thm8.1}
 The asymptotic representation  $\ Q_n(z)\ =\ W_n(z)(1+\epsilon_n(z)) \quad$  is valid for $z\in\mbb C\sim S\ $
where $\epsilon_n(z)$ is small on compacts in $\mbb{C}\sim S$ apart from zeros of $h_n(z)$ and
\begin{gather*}
W_n(z)=C_n f(z)^{-1/2}\frac{\sqrt{h_n(z)}}{(AV)^{1/4}(z)}\ e^{H_n(z)}, \\
H_n(z)=\(n+\frac12\)G(z)-\frac12\sum_{k=1}^{p-2}
s_{k,n}G\(z,{\zeta}_{k,n}\)+\sum^{p-2}_{k=1}\Delta_{k,n}\Omega_k(z)\ :
\end{gather*}
parameters $\ \Delta_{k,n}\in \mbb C$ , $\ {\zeta}_{k,n}\in\Omega,\ s_{k,n}=\pm 1 $ are defined by a system of equations
$$
\int_{C_j}dH_n=\gamma_j(f)\mod[Periods],\quad j=1,\dotsc,2p-4
$$
where $C_j$ is a collection of cycles in $\ \mbb C\sim A$ obtained by projectiong a homology basis on $\mc{R}$ onto the plane  ($\gamma_j$ depend on $f$ and the choice of cycles).
\end{theorem}

%The form of the asymptotics in the theorem is still similar to the Akhiezer - Widom formula but the nature of the situation is more complicated. First the for same number of branch points we have twice bigger genus of associated Riemann surface (which is equal to the number of unknown parameters).

The form of the asymptotic formula in the theorem is still somewhat similar to the form of Akhiezer--Widom formula, but the nature of the situation is more complicated. First, for the same number of branch points we have twice larger genus of associated Riemann surface which is equal to the number of unknown parameters. 
Second, set of parameters  contains two subsets of different nature; accordingly system of equations on parameters (representing conditions on periods of $H_n$ over cycles $C_j$) is a combination of  a system of equations on periods of first kind integrals with a JIP (this JIP has to be actully stated as a problem on $\mc R$).

Finally, in place of $\Delta_{k,n}$ one can use another set of $p-2$ parameters  $v_{k,n}$ which may be called "effective Chebotarev's centers"  More exactly, streightforward application of LG produce in place of the  fixed $V$ (which came from Chebotarev's problem)  a polynomial $V_n(z)=\prod_{k=1}^{p-2}(z-v_{k,n})$ depending on $n$ and, subsequently - variable Riemann surface ${\mc R}_n$. It turns out  that $V_n$ converges to $V$ and we can define  $\Delta_{k,n}= n(v-v_{k,n})$ (those numbers are proved to be bounded).  So, there is more than one choice of parameters.
Here we have observed one of the differences between methods of LG and MRH. The latter require that we make all the choices for all parameters; the former suggests some choices. % with all related advantages and disadvantages.

Now we briefly outline a construction leading to electrostatic interpretation of the function $W_n(z)$ above. It has some resemblance to what we have seen in Section \ref{sec7} for the real case but it is based on a more sophisticated equilibrium problem. In place of the equilibrium on a compact set we will have an $S$-equilibrium problem; existence and characterization of its solution follow by theorem \ref{thm3.1}. The construction is naturally based on zeros of a variable poyinomial $V_n$ and $h_n$ as basic parameters. Actually, twice more parameters are involved in analytic desciption of situation (see polynomials $A_n, \tilde h_n$ below) but the other parameters are determined by basic ones.

We fixed  $n\in \mbb N$ and set of parameters $\zeta_{k,n}\in\Omega=\ol{\mbb{C}}\sim S$ and $s_{k,n}=\pm 1$ we define charges $\sigma_n$  and $\nu_n$ by \eqref{eq7.11}. Then we define external field  external fields $\varphi_n$ (compare \eqref{eq7.12}) and polynomials $h_n$
$$
\varphi_n(z)\ =\ \sum_{k=1}^p\(\frac{\alpha_k}2+\frac14\)\log\frac1{\left|z-\alpha_k\right|}\ + \  V^{\sigma_n}(z)\ ; \quad
h_n(z) =\prod_{k=1}^{p-2} (z-\zeta_{k,n})
$$
% \frac12\ \sum_{j=1}^{p-2}s_{j,n}\log\frac1{\left|z-\zeta_{k,n}\right|}
Let  $\mc{T}$ be a family of continua in $\mbb{C}$ connecting points in $A=\left\{a_1,\dotsc,a_p\right\}$ (if $\alpha_j<0$ then $\Gamma\in\mc{T}$ goes around $a_k$ by a small loop): let $l_n$  be the number of positive $s_{j,n}$. Then theorem \ref{thm3.1} implies that there exists a compact $S_n\in\mc{T}$ with $S$-property in the field $\varphi_n$ relative to measures with total mass $n-l_n$.

Thus, $S_n$ is uniquely defined by $n\in \mbb N$ and set $\left\{\zeta_{k,n}, s_{k,n},\ k= 1, \dots, p-2\right\}$. Further, by theorem \ref{thm3.1}
$S_n$ is a union of $2p-3$ Jordan analytic arcs $\ S_{n,j\ }$ and, may be, some number of loops $L_{n,j}$,  all are trajectories of quadratic differential 
%We note that under our assumption that $S$ is a continuum \tu{(}and there is no cancellation\tu{)}
%\tu{(}surrounding $a_k$ with $\alpha_k<0\tu{)}$,
% $l_n=\frac12\sum\limits_{k=1}^{p-2}\(1+s_{j,n}\)$.
$$
                  - R_n(z) (dz)^2\ , \quad  \quad     R_n(z) =  C_n\ \{A_n(z)V_n(z){\tilde h_n(z)}^2 \}/\{A(z)^2 h_n(z)^2\}
$$
with polynomials satisfying $\ A_n \to  A,\ V_n \to V ,\ \tilde h_n -h_n \to 0\  $ as $\ n \to \infty$. Further, each $\ S_{n,j}\ $ connect a root of $V_n$ with a root $a_{j,n} $ of  $A_n\ $ (or, two roots of $V_n$); for large $n$ we have $a_{j,n}\approx a_j \ $ where $a_j$ is a root of $A$ closest to $a_{n,j}$. If for this $j$ we have $\alpha_j <0$ then there is a loop $L_{n,j}\in S$  around $a_j$ containing $a_{j,n}$. In this case we define $\tilde S_{n,j} = S_{n,j}\cup L_{n,j} $; otherwise (if $\alpha_j>o$) there is no loop in $S_n$ around $a_j$  and we set $\tilde S_{n,j}= S_{n,j}$ and we do same if  $S_{n,j}$ connects two roots of $V_n$.

Finally, for large enough $n$ the compact  $\ S_n\ $ coinside with support of its equilibrium measure $\lambda_n$.
\begin{theorem}\label{thm8.2}
There exists a set of parameters  $\left\{\zeta_{k,n},s_{k,n},\ k= 1, \dots, p-2\right\}$ such that for the eqilibrium measure $\lambda_n$ of associated $S$-compact  $S_n$ we have
$$ \lambda_n (\tilde S_{n,j}) \in \mbb Z, \quad j=1, \dots, 2p-4 \ ; \quad  W_n(z) = \exp\{-\mc V^{\lambda_n+\sigma_n}\}$$
\end{theorem}
We do not go here into any further details.

% Section 9
\section{Proof of the existence theorem}\label{sec9}

% Sec. 9.1
\subsection{Extremal energy functional}\label{subsec9.1}

We consider finite positive Borel measures $\mu$ in the extended complex plane  $\ol{\mbb{C}}$ satisfying condition
\begin{equation}
\label{eq9.1}
\int\log\(1+|x|\) d\mu(x)<+\infty.
\end{equation}
Then the logarithmic potential $V^{\mu}(z)=\int\log\frac1{|z-x|}d\mu\in(-\infty,+\infty]$ exists at any $z\in\mbb{C}$ and a.e.\  finite. Let $\mc{M}$ be the set of all such measures with finite energy $\mc{E}(\mu)=\int V^\mu(z)\,d\mu(z)<+\infty$ (in general $\mc{E}(\mu)\in(-\infty,+\infty]$).

Let $\mc{K}$ be (here) the family of all compacts $K\subset\ol{\mbb{C}}$ with finite number of (connected) components; at least one of them is not a single point.

For $K\in\mc{K}$ we denote by $C^\ast(K)$ the set of all real valued functions $\varphi$ which are continuous on $K$ except for a finite set of points $e$, depending on $\varphi$; there are no restrictions on the behavior of $\varphi(z)$ as $z\to z_0\in e$. Denote
$$
\mc M(K,\varphi)=\left\{\mu\in\mc{M}:\mu\in L_1(\mu),|\mu|=1,S_\mu\subset K\right\}.
$$
We define total energy of $\mu\in\mc{M}(K,\varphi)$ and equilibrium energy of $K\in \mc K$
$$
\mc{E}_\varphi(\mu)=\mc{E}(\mu)+2\int\varphi\,d\mu;\quad 
\quad\mc{E}_\varphi[K]=\inf_{\mu\in\mc{M}(K,\varphi)}\mc{E}_\varphi(\mu).
$$

In the next sec. \ref{subsec9.2} we consider transformation of potentials and energies under linear fractional mappings of $\ \ol{\mbb{C}}$. This would allow us to reduce considerations of compacts in $\ol{\mbb{C}}$ and measures on them to consideration of compacts and measures in ${\mbb{C}}$
(and, subsequently, if it  is convenient, to the case $K\subset \ol{D}_{1/2}=\{z:|z|\leq 1/2\}$).  In sections \ref{subsec9.3}--\ref{subsec9.9} we use  $\mc{K}$  to denote class of compacts in the open plane (we return to spherical compacts in the end of the last subsection 9.10). 

In sec. \ref{subsec9.3} we prove that if $\mc{E}_\varphi[K]>-\infty$ then there exist a unique extremal measure 
$$
\lambda=\lambda_{\varphi,K}\quad{\text defined}\ {\text  by}\quad\mc{E}_\varphi(\lambda)=\mc{E}_\varphi[K].
$$
Next sec. \ref{subsec9.4}  contains well known definitions related to balayage needed for further references. Then in sections
 \ref{subsec9.5} -- \ref{subsec9.9} we investigate continuity properties of the extremal energy functional
$$
\mc{E}_\varphi:\mc{K}\to[-\infty,+\infty)
$$
which is the main subject of section 9.  More exactly,  in \ref{subsec9.5} and  \ref{subsec9.6}  we prove continuity of $\mc{E}_\varphi$ on class of compacts in plane without components of small capacity; we obtain some explicit estimates. In sec. \ref{subsec9.7} we prove continuity of $\mc{E}_\varphi$ on class $\mc K_s$ of compacts with bounded number of components in a continuos external field  $\varphi$.; in sec. \ref{subsec9.8} we extend this result to bounded $\varphi$.

In sec. \ref{subsec9.9} we prove  that $\mc{E}_\varphi[K]$ is upper semi-continuous under general assumptions. Sec. \ref{subsec9.10} contains conclusion of the proof of theorem \ref{thm3.1}.

% Sec. 9.2
\subsection{Linear fractional transformations of potentials}\label{subsec9.2}

For a fixed $\zeta_0\in\mbb{C}_\zeta$, $\ c\in\mbb{C}\ $ consider  a mapping of this $\zeta$-plane onto a $z$-plane and corresponding transformation of measures -- measure $\mu$ in the $z$-plane and measure  $\nu$ in the $\zeta$-plane 
\begin{equation}
% 9.2
\label{eq9.2}
z=\frac c{\zeta-\zeta_0}\ ,\quad \zeta=\zeta_0+\frac cz\ ; \quad \quad\quad d\mu (z) = d\nu (\zeta)\ .
\end{equation}

If $\ \nu\in\mc{M}$ and $V^\nu\(\zeta_0\)<+\infty$  then  $\mu$ defined  by \eqref{eq9.2} also belong to $\mc M$. 
We have the following relations potentials of $\mu$ and $\nu$
\begin{align}
% 9.3
\label{eq9.3}
V^\mu(z) &=V^\nu(\zeta)-V^\nu\(\zeta_0\)+\log\left|{(\zeta-\zeta_0)} / {c} \right|; \\
V^\nu(\zeta) &=V^\mu(z)+V^\mu(0)+\log|z|+\log |c|.
\label{eq9.4}
\end{align}
Indeed, making the substitution of $x=c /(t-\zeta_0)\ $ in $\ V^\mu(z)=-\int\log|z-x|\,d\mu(x)$ and using identity
$$
z-x=\frac{c}{\zeta-\zeta_0}-\frac{c}{t-\zeta_0}
=\frac{c(t-\zeta)}{\(t-\zeta_0\)\(\zeta-\zeta_0\)}
$$
we come to \eqref{eq9.3}; \eqref{eq9.4} follows since $\log\left|{(\zeta-\zeta_0)} / {c}\right|=-\log|z|$ and $V^\nu\(\zeta_0\)=-V^\mu (0) -\log|c|$. Integrating \eqref{eq9.3} with $d\mu(z)=d\nu(\zeta)$ we come to
% 9.5
\begin{equation}
\label{eq9.5}
\mc{E}(\mu)=\mc{E}(\nu)-2V^\nu\(\zeta_0\)-\log|c|.
\end{equation}

Next,  let a weight $\psi\in C^\ast\(\mbb{C}_\zeta\)$ is given in the $\zeta$-plane. We introduce an associated weight in the $z$-plane by
\begin{equation}
\label{eq9.6}
\varphi(z)=\psi(\zeta)-\log\left|\zeta-\zeta_0\right|+\log|c|.
\end{equation}

% Lemma 9.1
\begin{lemma}\label{lem9.1}
For $z$ and $\zeta$ connected by \eqref{eq9.2}, measures $d\mu(z)=d\nu(\zeta)$ and external fields  $\varphi$ and $\phi$ related by \eqref{eq9.6}  we have
\begin{align*}
\(V^\mu+\varphi\)(z) &=\(V^\nu+\psi\)(\zeta)-V^\nu\(\zeta_0\) \\
\mc{E}_\varphi(\mu) &=\mc{E}_\psi(\nu)+\log|c|.
\end{align*}
\end{lemma}

\begin{proof}
Integration \eqref{eq9.6} with $d\mu(z)=d\nu(\zeta)$ yields
$$
2\int\varphi\,d\mu=2\int\psi\,d\nu
+2V^\nu\(\zeta_0\)+2\log|c|.
$$
By adding \eqref{eq9.5}, we obtain a second assertion of lemma \ref{lem9.1}. The first assertion is obtained by adding up  \eqref{eq9.3} and \eqref{eq9.6}.
\end{proof}

Let $K\in\mc{K}$, $K\subset\ol{\mbb{C}}_\zeta$ and $\zeta_0\not\in K$. Then for any $\nu\in\mc{M}(K)$ we have $V^\nu\(\zeta_0\)<+\infty$. Let $\wt{K}\in\mbb{C}_z$ be the image of $K$ under linear fractional transformation \eqref{eq9.2}. If $|c|$ is small enough then
$$
\wt{K}\subset\ol{D}_{1/2}=\{z:|z|\le 1/2\}.
$$

Let $\psi\in C^\ast(K)$ and $\varphi\in C^\ast (\wt{K} )$ be related by \eqref{eq9.6} and \eqref{eq9.2}. It follows by Lemma \ref{lem9.1} that there is one-to-one correspondence between measures $\nu\in\mc{M}(K,\psi)$ and $\mu\in\mc{M} (\wt{K},\varphi )$ defined by $d\mu(z)=d\nu(\zeta)$. Furthermore, we have
$$
\mc{E}_\varphi [\wt{K} ]=\mc{E}_\psi[K]-\log|c|
$$
and if an extremal measure exists for one of the two compacts then it also exists for another one; extremal measures are related by the same equation $d\mu(z)=d\nu(\zeta)$.

It is not difficult to verify that transformation \eqref{eq9.2} also preserves the $S$-property, but we are not using this fact.

% Sec 9.3
\subsection{Equilibrium measure}\label{subsec9.3}

We fix $K\in\mc{K}$, $\varphi\in C^\ast(K)$.

\begin{theorem}
If $\mc{E}_\varphi[K]>-\infty$ then there exists a unique measure $\lambda=\lambda_{\varphi,K}\in\mc{M}(K,\varphi)$ with
$$
\mc{E}_\varphi(\lambda)=\mc{E}_\varphi[K]
$$
-- equilibrium measure for $K$ in the external field $\varphi$.
\end{theorem}

\begin{proof}
We prove the theorem under an additional technical assumption that $K\ne\ol{\mbb{C}}$, so that $K$ has an exterior point $\zeta_0$ in $\mbb{C}$ (we are going to apply the theorem under this restriction; it is, however, not necessary).

By the remark in Section \ref{subsec9.2} above, we may assume without loss of generality that $K\subset\ol{D}_{1/2}$. Then for any $\mu\in\mc{M}=\mc{M}(K,\varphi)$ we have $V^\mu(z)\ge 0$ on $K$, $\mc{E}(\mu)$ is a norm on $\mc{M}$ and positive Borel measures on $K$ constitute a complete metric space with this norm.

Let $\mu_n\in\mc{M}$ be a minimizing sequence for $\mc{E}_\varphi$, that is $\mc{E}_\varphi\(\mu_n\)\to\mc{E}=\mc{E}_\varphi[K]$. It follows from the identity
$$
\mc{E}(\mu-\sigma)=2\mc{E}_\varphi(\mu)+2\mc{E}_\varphi(\sigma)-4\mc{E}_\varphi\(\frac{\mu+\sigma}2\)
$$
that $\mu_n$ is a Cauchy sequence in metric of energy: $\mc{E}\(\mu_n-\mu_m\)\to 0$ as $n,\  m\to\infty$. Therefore there is $\mu$ on $K$ with finite energy such that $\mc{E}\(\mu_n-\mu\)\to 0$ as $n\to\infty$. As a corollary we have $\ \mc{E}\(\mu_n\)\to\mc{E}(\mu)\ $ and also weak-star convergence $\mu_n\overset{\ast}{\to}\mu$.
Next, we will prove that $\varphi\in L_1(\mu)$, therefore, $\mu\in\mc{M}$ and, subsequently, that $\mu = \lambda_{\varphi, K}$. The proof is based on the following assertion stated for the defined above extremal measure  $\mu$. It is important to motice that the assertion of lemma \ref{lem9.3} is valid for the equilibrium measure $\mu= \lambda_{\varphi, K}$ of any compact  $K\subset\ol{D}_{1/2}$.
%\textbf{Lemma 9.3}
\begin{lemma}\label{lem9.3}
For the extremal measure $\mu$ above and for constant $M$ defined by
\begin{equation}
\label{eq9.7}
M = 1 + \inf_{\epsilon>0}\(\max_{z\in K_{\epsilon}}\varphi(z)+2\gamma\(K_\epsilon\)\)
\end{equation}
where $ K_{\epsilon}=K\sim (e)_{\epsilon},\quad (e)_{\epsilon}=\{z\in K:\dist(z,e)<\epsilon\}$ , we have
\begin{equation} \label{eq9.7.1}
\mu(\{z\in K\smallsetminus e:\varphi(z)>M\})=0.
\end{equation}
\end{lemma}

\begin{proof}
Since $\varphi\in C(K\sim e)$ the number $M_{\epsilon}=\max\limits_{K_{\epsilon}}\varphi\ $ is finite. If $\epsilon>0$ is small enough then $K_{\epsilon}$ is not empty and of positive capacity, so that its Robin measure $\omega_{\epsilon}$ has finite energy and the Robin constant $\gamma_{\epsilon}$ is finite.
For such  $\ \epsilon>\ 0$  we define $M=M_{\epsilon}+2\gamma_{\epsilon}\ $, $\wt{K}=\{z\in K\sim e:\varphi(z)>M\}$, and,  then,
$$
\wt{\mu}=\left.\mu_n-\mu_n\right|_{\wt K}+t_n\omega_{\epsilon},
\quad t_n=\mu_n(\wt{K}).
$$

We have $\mc{E}\(\wt{\mu}_n\)\le\mc{E}\(\mu_n+t_n\omega_{\epsilon}\)\le\mc{E}\(\mu_n\)+2t_n\int V^{\omega_{\epsilon}}d\mu_n+t_n^2\int V^{\omega_{\epsilon}}d\omega_{\epsilon}\le\mc{E}\(\mu_n\)+3t_n\gamma_{\epsilon}$ (note that $t_n\le 1$ and $V^{\omega_{\epsilon}}\le\gamma_{\epsilon}$ in $\mbb{C}$). On the other hand, $\int\varphi\,d\wt{\mu}_n\le\int\varphi\,d\mu-Mt_n+M_{\epsilon}t_n$. Adding this inequality doubled to the preceding one, we obtain
$$
\mc{E}_\varphi\(\mc{\mu}_n\)\le\mc{E}_\varphi\(\mu_n\)-t_n\varphi_{\epsilon}<\mc{E}_\varphi\(\mu_n\).
$$
It follows that $\wt{\mu}_n$ is also a minimizing sequence, that is $\mc{E}_\varphi\(\wt{\mu}_n\)\to\mc{E}$. Therefore, it converges in energy (and weak-star) to the same measure $\mu$. Since $\wt{K}$ is a relative open subset of $K$ we have $\mu (\wt{K} )\le\varliminf\mu_n (\wt{K} )=0$. Proof of lemma \ref{lem9.3} is completed.
\end{proof}

It follows by Lemma \ref{lem9.3} that $\int\varphi^+\,d\mu<+\infty$ for $\varphi^+=\max\{\varphi,0\}$. On the other hand, $\mc E>-\infty$ implies $\varliminf\int\varphi^-d\mu_n>-\infty$. Hence, $\int\varphi^-d\mu>-\infty$, and so $\varphi\in L_1(\mu)$.

It remains to prove that $\mc{E}_\varphi(\mu)=\mc{E}=\lim\limits_{n\to\infty}\mc{E}_\varphi\(\mu_n\)$. If it is not so then there exist $\delta>0$ and an infinite sequence $\Lambda\subset\mbb{N}$ such that

% 9.9
\begin{equation}
\label{eq9.8}
\int\varphi\,d\mu_n\le\int\varphi\,d\mu-3\delta,\quad n\in\Lambda.
\end{equation}
We will show that such an assumption allows us to construct a new sequence of measures $\wt{\mu}_n\in\mc{M}$ with $\lim\mc{E}_\varphi\(\wt{\mu}_n\)<\mc{E}$ (actually we can make it $-\infty$) in contradiction with our original assumptions.

Let $\epsilon\in[0,1]$ and $\partial e_{\epsilon}$ be the boundary of $e_{\epsilon}=\left\{z\in K:\dis(z,e)>\epsilon\right\}$. Let $E=\left\{\epsilon:\mu\(\partial e_{\epsilon}\)=0\right\}$ then $[0,1]\sim E$ has Lebesgue measure zero. For a fixed $\epsilon\in E$ denote
$$
\nu_n=\left.\mu_n\right|_{e_{\epsilon}},\quad
\nu=\left.\mu\right|_{e_{\epsilon}},\quad
\sigma_n=\mu_n-\nu_n,\quad
\nu=\sigma-\nu.
$$
We have $\int\varphi\,d\sigma_n\to\int\varphi\,d\sigma$. Combined with \eqref{eq9.8}, it implies
$$
\int\varphi\,d\nu_n<\int\varphi\,d\nu-2\delta,\quad n\in\Lambda_1.
$$
Next, $\nu=\left.\mu\right|_{e_{\epsilon}}$ is continuous, $\varphi\in L_1(\mu)$, therefore $\int\varphi\,d\nu\to 0$ as $\epsilon\to 0$, $\epsilon\in E$. We may conclude that there exist a sequence $\epsilon_n$, $n\in\Lambda_2$, with $\epsilon_n\in E$, $\epsilon_n\to 0$ such that
$$
\int_{e_{\epsilon_n}}\varphi\,d\mu_n<-\delta,\quad n\in\Lambda_2.
$$

Finally, we introduce a new sequence of measures
$$
\wt{\mu}_n=\left.2\mu_n\right|_{e_{\epsilon_n}}+\left.\(1-t_n\)\mu_n\right|_{K\sim e_{\epsilon_n}}\in\mc{M}
$$
where $t_n=\mu_n\(e_{\epsilon_n}\)\to 0$ as $n\to\infty$, $n\in\Lambda_2$. It is clear that $\mc{E}\(\wt{\mu}_n\)-\mc{E}\(\mu_n\)\to 0$ as $n\to\infty$ and at the same time $\int\varphi\,d\wt{\mu}_n\le\int\varphi\,d\mu-\delta$, $n\in\Lambda_2$. From here $\mc{E}_\varphi\(\wt{\mu}_n\)\le\mc{E}_\varphi\(\mu_n\)-\delta$ is in contradiction with the assumption that $\mu_n$ is minimizing for $\mc{E}_\varphi$.
\end{proof}

% Sec.9.4
\subsection{Balayage}\label{subsec9.4}

Let $K\in\mc{K}$ and $\Omega$ be a component of $\ol{\mbb{C}}\sim K$; here we assume that $K\subset\mbb{C}$ and $\infty\in\Omega$. Let $g(z,\zeta)$ be the Green function for $\Omega$ with a pole at $\zeta\in\Omega$. Let $\mu\in \mc M$ be a measure in $\ol{\Omega}$, we denote (as in sec. \ref{subsec6.1})
$$
V_{\Omega}^{\mu}(z)=\int g(z,\zeta)\,d\mu(\xi)
$$
--- the Green's potential of $\mu$. By $\ \wt{\mu}\ $ we denote the balayage of $\mu$ onto $\partial\Omega\subset K$. We note that condition $\supp\mu\subset\ol{\Omega}\ $ is not necessary; one can assume that balayage does not affect part of $\mu$ in the complement to $\Omega$.

% Lemma 9.4             $\mu\in\mc{M}\(\ol{\Omega}\)$
\begin{lemma}\label{lem9.4}
We have for $\ z\in\ol{\Omega}$
\begin{align}
\label{eq9.9}
V^{\wt{\mu}}(z) &=V^\mu(z)-V_{\Omega}^{\mu}(z)+\int g(\zeta,\infty)\,d\mu(\xi) \\
\mc{E}\(\wt{\mu}\) &=\mc{E}(\mu)-\mc{E}^\Omega(\mu)+2\int g(\zeta,\infty)\,d\mu(\xi)
\label{eq9.10}
\end{align}
where $\mc{E}^\Omega(\mu)=\int V_{\Omega}^{\mu}(\zeta)\,d\zeta$ --- Green's energy of $\mu$.
\end{lemma}
\begin{proof}
Denote $C=\int g(\zeta,\infty)\,d\mu(\zeta)$. The function
$$
V(z)=\int\(\log\frac1{|z-\zeta|}-g(z,\zeta)\)d\mu(\xi)+C
$$
in the right-hand side of \eqref{eq9.9} is harmonic in $\Omega$ and it is equal to $V^\mu(z)+C$ on $\partial\Omega$. As $z\to\infty$ we have $V(z)+\log|z|\to 0$ by the symmetry of Green's function. These are conditions uniquely defining $V^{\wt{\mu}}$ and \eqref{eq9.9} follows.

Next, the function $V^\mu(z)-V_{\Omega}^{\mu}(z)+g(z,\infty)$ is harmonic in $\Omega$; using \eqref{eq9.9} we obtain
\begin{equation*}
\begin{split}
\mc{E}\(\wt{\mu}\)=\int V^{\wt{\mu}}d\wt{\mu}
=\int\(V^\mu-V_{\Omega}^{\mu}+C\)d\wt{\mu} 
=\int\(V^\mu-V_{\Omega}^{\mu}+g+C\)d\wt{\mu}= \\
=\int\(V^\mu-V_{\Omega}^{\mu}+g+C\)d\mu 
=\mc{E}(\mu)-\mc{E}^\Omega(\mu)+\int g\,d\mu+\int g\,d\mu.
\end{split}
\end{equation*}
\eqref{eq9.10} follows and the proof is completed.
\end{proof}

Let $K\in\mc{K}$ and $\Omega$ be a component of $\ol{\mbb{C}}\sim K$. We assume now that infinity is not in $\Omega$; domain may be bounded or not. Again, we consider balayage of a measure $\mu\in\mc{M}$ onto $\partial\Omega$. Assertions of Lemma \ref{lem9.4} have to be slightly modified.
This leads to the next lemma whose proof s similar to the proof of lemma \ref{lem9.4}. Ssituation is actually even simplier since $V^{\mu}_{\Omega}=0$ on $\partial\Omega$.

% Lemma 9.5
\begin{lemma}\label{lem9.5}
Let $\mu\in\mc{M}$, $\wt{\mu}$-balayage of $\mu$ onto $\partial\Omega$, $z\in\ol{\Omega}$, $\infty\not\in\Omega$. Then
\begin{align}
\label{eq9.11}
&V^{\wt{\mu}}(z) =V^\mu(z)-V_{\Omega}^{\mu}(z)\\
&\mc{E}\(\wt{\mu}\)=\mc{E}(\mu)-\mc{E}^\Omega(\mu)
\label{eq9.12}
\end{align}
\end{lemma}

%  ($\ j=1,\dotsc,m$ or $j\in\mbb{N}$)
In general, for a compact $K\subset\mbb {C}$ the complement $\Omega=\ol{\mbb{C}}\sim K$ is a union of finite or countable number of components $\Omega_j$ . Let $\mu\in\mc{M}(\mbb{C})$ we define $ \mu_j= \mu |_{\Omega_j}$. Further, let $\wt{\mu}_j$ be the balayage of $\mu_j$ onto $\partial\Omega_j$. 

We define the balayage $\wt{\mu}$ of $\mu$ onto $K$ by $\wt{\mu}=\sum_j\wt{\mu}_j,\quad$

The following lemma is a  combonation of lemma \ref{lem9.4} and lemma \ref{lem9.5}; note also that $V^{\wt{\mu}_j}(z) =V^{\mu_j}(z)$ for $z\in \mbb C\sim\Omega_j$, $\infty\not\in\Omega_j$.

% Lemma 9.6
\begin{lemma}\label{lem9.6}
In notations and settings above we have
$$
\mc{E}\(\wt{\mu}\) =\mc{E}(\mu)- \sum_j\mc{E}^{\Omega_j} (\mu_j)  +2\int g(\zeta,\infty)\,d\mu(\xi)
$$
where summation is taken over all components of $\ \ol{\mbb{C}}\sim K$ and $g(z,\infty)$ is the Green function for the component containing $\infty$ \tu{(}if there is no such component then the integral term is dropped\tu{)}.
\end{lemma}

%Sec. 9.5
\subsection{Lemma on harmonic extension}\label{subsec9.5}

In this (and next) section we consider set $\mc K$ of compacts $K\subset\mbb{C}$ without small components. More exactly, we assume that  for some $c>0$  capacity of any connected component of $K\in\mc{K}$ is at least $c>0$ and $c$ is same for all compacts (thus, we actually have $\mc K = \mc K (c)$).

Next, let $\Omega$ be a domain in $\mbb{C}$ containing $K$ and $\varphi\in C\(\ol{\Omega}\)$ be a real-valued function (external field) with modulus of continuity $\omega(\delta)$ in $\ol{\Omega}$.

Finally, let $\wt{\varphi}(z)$ be the harmonic extension of $\varphi(z)$ from $\ K\ $ to $\ \ol{\mbb{C}}\ $; that is $\ \wt{\varphi}(z) = \varphi(z)\ $ on $K$, $\wt{\varphi}(z)$ is harmonic in any connected component $G$ of $\ol{\mbb{C}}\sim K$ (and $\ \wt{\varphi}=\varphi\ $ on $\ \partial G\ $). We do not try here to make general settings: actually we need only a bounded harmonic extention of  $\varphi$ to some neighbouhood of $K$, say, too $\Omega$; assertion of the next lemma wil not change.

% Lemma 9.7
\begin{lemma}\label{lem9.7}
Under the assumptions above for any $z\in\mbb{C}$ satisfying
% 9.14
\begin{equation}
\label{eq9.13}
\delta = dist(z, K) \leq dist(z, \partial\Omega); \quad \pi\delta \leq c
\end{equation}
we have
% 9.15
\begin{equation}
\left|\varphi(z)-\wt{\varphi}(z)\right|
\le 4\(M\sqrt{\delta}+\omega(\delta)\)
\end{equation}
where $\ M=\max\limits_K|\varphi|\ $ and $\ \dist(E,F)\ =\min\limits_{x\in F,y\in F}|x-y|$.
\end{lemma}

\begin{proof}
For a fixed $z$ satisfying  \eqref{eq9.13}, let $x\in K$ a closest to $z$ point on $K$, that is $\delta=|z-x|$; the interval $(z,x)$ belongs to $\Omega$. Let $r\le\sqrt{\delta}$ be such that $D=\{\zeta:|\zeta-x|<r\}\subset\Omega$.

Let $F$ be a connected component of $K$ containing $x$; then $\partial D\cap F\ne\emptyset$. Indeed, the contrary would imply that $F\subset D$ then $\cop(F)<\pi r^2\le\pi\delta$ in contradiction with $\cop(F)\geq c\ge\delta\pi$ (see \eqref{eq9.13}).

Let $\mc{D}$ be a connected component $D\sim K$ with $\(z, x\)\subset\mc{D}$.

Denote $u(\zeta)=\wt{\varphi}(\zeta)-\varphi(x)$; then we have
% 9.16
\begin{equation}
\label{eq9.14}
\left|\varphi(z)-\wt{\varphi}(z)\right|
\le|\varphi(z)-\varphi(x)|+\left|\wt{\varphi}(z)-\varphi(x)\right|
\le\omega(r)+|u(z)|.
\end{equation}
If $\partial\mc{D}\subset F$ then by the maximum principle for $u$ and the equality $\varphi=\wt{\varphi}$ on $F$ we obtain
$$
|u(z)|\le\max_{\zeta\in F}\left|\wt{\varphi}(\zeta)-\varphi(x)\right|
\le\max_{\zeta\in F}|\varphi(\zeta)-\varphi(x)|\le \omega(r)
$$
and the assertion of the lemma follows (for $r=\sqrt{\delta})$. Otherwise, there exists a continuum $\gamma_0\subset\partial\mc{D}$ connecting $x\in\partial\mc{D}$ and $\partial D$.

Let $\gamma_1=\partial\mc{D}\sim\gamma_0$ be the rest of the boundary $\partial\mc{D}$ of $\mc{D}$. Denote by $h(\zeta)$ the harmonic function in $\mc{D}$ with boundary values $h=0$ on $\gamma_0$, $h=1$ on $\gamma_1$ (harmonic measure of $\gamma_1\subset\partial\mc{D}$ relative to $\mc{D}$). We have $|u(\zeta)|\le 2M$ on $\gamma_1$ and $|u(\zeta)|\le\omega(r)$ on $\gamma_0$, thus, by maximum principle

% 9.17
\begin{equation}
\label{eq9.15}
|u(z)|\le 2Mh(z)+\omega(r)(1-h(z))
\le 2Mh(z)+\omega(r).
\end{equation}

Next, let $\wt{\mc{D}}=D\sim\gamma_0$ (it is an extension of $\mc{D}$ obtained by moving out part of $\partial\mc{D}$ where $h=1$). Let $\wt{h}(\zeta)$ be harmonic in $\wt{\mc{D}}$, $\wt{h}=0$ on $\gamma_0$ and $\wt{h}=1$ on $\partial\wt{\mc{D}}\sim\gamma_0$. By the maximum principle we have $h(\zeta)\le\wt{h}(\zeta)$, $\zeta\in\mc{D}$. Now, by H. Milloux's theorem (see \cite{Gol66}, Chapter VIII, Theorem 6) the maximum value for $\wt{h}(z)$ at the fixed point $z\in\wt{\mc{D}}$ is reached if $\wt{\mc{D}}=D\sim\Delta$ where $\Delta$ is the radius of $D$ opposite to $z$ and we have
$$
h(z)\le\wt{h}(z)\le\frac2\pi\arcsine\frac{1-\delta/r}{1+\delta/r}<2\sqrt{\delta/r}.
$$
Combining this with \eqref{eq9.14} and \eqref{eq9.15}, we obtain the assertion of the lemma.
\end{proof}

% Sec. 9.6
\subsection{Continuity of $\mc{E}_\varphi$ for continuous $\varphi$ on compacts without small components}\label{subsec9.6}

Under assumptions of sec 9.5 above we prove an explicit estimate for $\left|\mc{E}_\varphi\(K_1\)-\mc{E}_\varphi\(K_2\)\right|$. More exactly, let $\Omega$ be a domain in $\mbb{C}\ $, $\varphi\in C\(\ol{\Omega}\)$ and $\omega(\delta)$ be the modulus of continuity of $\varphi$ in $\ol{\Omega}$; assume $|\varphi|\le M$ in $\ \ol{\Omega}\ $. Let $K_1,\ K_2$ be two compacts in $\Omega$; assume that each connected component of each of the two compacts has capacity of at least $c>0$.

% Theorem 9.8
\begin{theorem}\label{thm9.8}
Under the conditions above, if
$$
\delta=\dist\(K_1,K_2\)\le\dist\(K_i,\partial\Omega\); \quad  i=1,2, \quad\pi\delta\le C
$$
then we have
$$
\left|\mc{E}_\varphi\[K_1\]-\mc{E}_\varphi\[K_2\]\right|
\le 4\(\sqrt{\delta/c}+M\sqrt{\delta}+\omega\(\sqrt{\delta}\)\).
$$
\end{theorem}

Theorem \ref{thm9.8} is essentially a combination of lemma \ref{lem9.7} and the following estimate for Green's function which we prove first. 

% Lemma 9.9
\begin{lemma}\label{lem9.9}
Let $K\in\mc{K}$ be a compact in $\mbb{C}$ and let the capacity of each component of $K$ be at least $c>0$. Let $\Omega$ be a connected component of $\ \ol{\mbb{C}}\sim K$ with $\infty\in\Omega$ and let  $g(z)=g(z,\infty)$ be the Green function for $\Omega$. Then
$$
g(z)\le \sqrt{\delta/c},\quad
\delta=\dist(z,k),\quad z\in\ol{\Omega}.
$$
\end{lemma}

\begin{proof}
Let $x\in K$ be closest to the $z\in\ol{\Omega}$ point on $K$, that is,   $|z-x|=\delta$ and $F\subset K$ be a component of $K$ with $x\in F$. We denote  by  $\wt{\Omega}$  the component of $\ol{\mbb{C}}\sim F$ with $\infty\in\wt{\Omega}$ and by $\wt{g}(z)$ the Green function for $\wt{\Omega}$ with a pole at $\infty$. We have $\Omega\subset\wt{\Omega}$ and $g(z)\le\wt{g}(z)$ in $\Omega$. It is enough to prove that $\wt{g}(z)\le \sqrt{\delta/c}$; this inequality basically means that the maximum of $\wt{g}$ at a fixed $z$ is attained when $F$ is a segment of length $4c$ (capacity $c$) containing $z$ on its continuation with $\dist(z,F)=\delta$ (compare the end of the proof of Lemma \ref{lem9.7} ).

A convenient way to obtain a formal proof is to use standard estimates for conformal mappings.

Define $\Phi(z)=e^{G(z)}$ where $\wt{g}(z)=\Re G(z)$, $G\in H(\Omega)$ and normalizing constant for $\Im G$ is selected so that $G(x)=0$, $\Phi(x)=1$. As $z\to\infty$ we have $\Phi(z)= c/z+c_0+{c_1}z+\dotsb$; $\Phi$ is univalent in $\ol{\mbb{C}}\sim F=\Omega$ and maps this domain onto  $|\xi|>1$.
 Point  $x\in \partial\Omega$ is accessible, so, the inverse function $f = \Phi^{-1}$ which maps conformally  $|\xi|>1$ onto  $\Omega$ is (non-tangent)  continuous at $y=1$ is $\zeta$-plane. We have $F(\xi)=c\(\zeta+\alpha_0+{\alpha_1}/\zeta+\dotsb\)$ at $\infty$. Now by Theorem 1 in \cite{Gol66} (Chapter IV) we have $|F(\zeta)-F(y)|\ge c|\zeta-y|^2/\zeta$. From here with $\zeta=\Phi(z)$ and $x=\Phi(y)$ we obtain
$$
|\Phi(z)-1|^2\le\frac1c|z-x|
\cdot|\Phi(z)|\le\frac{\delta}c|\Phi(z)|.
$$
On the other hand, we have $|\Phi(z)-1|\ge|\Phi(z)|-1=e^{g(z)}$; from here and the above inequality
$$
g(z)^2\le\(e^{g(z)/2}-e^{-g(z)/2}\)^2\le\delta/c
$$
and lemma \ref{lem9.9} is proven.
\end{proof}

\begin{proof}[Proof of Theorem \ref{thm9.8}]
Let $\lambda$ be the equilibrium measure for $K_2$ and let $\mu$ be the balayage of $\lambda$ onto $K_1$. By Lemma \ref{lem9.5}
$$
\mc{E}_\varphi\[K_1\]\le\mc{E}_\varphi(\mu)
\le\mc{E}_\varphi(\lambda)
+2\int\(g+\wt{\varphi}-\varphi\)d\lambda
$$
where $g=g(z,\infty)$-Green's function for component of $\ol{\mbb{C}}\sim K_1$ containing $\infty$, $\wt{\varphi}$ is harmonic in each component of $\ol{\mbb{C}}\sim K_1$, $\wt{\varphi}=\varphi$ on $K_1$. By Lemma \ref{lem9.9} we have $g(z)\le\sqrt{\delta/c}$ on $\supp(\lambda)$ and by Lemma \ref{lem9.6} we have $\left|\wt{\varphi}-\varphi\right|\le 4\(M\sqrt{\delta}+\omega\(\sqrt{\delta}\,\)\)$ on $\supp\lambda$. Altogether it makes
$
\mc{E}\[K_1\]\le\mc{E}\[K_2\]
+4\(M\sqrt{\delta}+\sqrt{\delta/c}+\omega\(\sqrt{\delta}\)\).
$
Since conditions on $K_1, K_2$ are symmetric, lemma \ref{thm9.8} follows.
\end{proof}

% Sec. 9.7
\subsection{Continuity of $\mc{E}_\varphi$ for continuous $\phi$ on compacts with bounded number of components}\label{subsec9.7}

Theorem \ref{thm9.8} asserts that for a continuous external field $\varphi$ in $\mbb{C}$, the functional of equilibrium energy $\mc{E}_\varphi[K]$ is continuous on the space of compacts $K\subset\mbb{C}$ with $\cop(F)\ge c>0$ for each connected component $F$ of each compact $K$. We did not assume that the number of components is bounded.

Now we drop the condition $\cop(F)\ge c>0$ on components of $K$ and introduce a restriction on the number of components. For a fixed $s\in\mbb{N}$ we recall notation $\mc{K}_s$ for  the set of all compacts $K\in \mbb {C}$ with at most $s$ componennts and  positive capasity. We will prove that for continuous $\varphi$ the functional of weighted equilibrium energy is continuous on $\mc{K}_s$ with Hausdorf metric. We do not include explicit estimates.  

% Theorem 9.10
\begin{theorem}\label{thm9.10}
For a continuous $\varphi(z):\mbb{C}\to\mbb{R}$ the functional $\mc{E}_\varphi[K]:\mc{K}_s\to(-\infty,+\infty)$ is continuous.
\end{theorem}

\begin{proof}
We need to show that for any $K_0\in\mc{K}_s$ and any $\epsilon>0$ there exist $\delta>0$ such that for any $K\in\mc{K}_s$ we have
\begin{equation}\label{eq9.16}
\left|\mc{E}_\varphi[K]-\mc{E}_\varphi\[K_0\]\right|<\epsilon\quad for\quad \delta_H(K,K_0)<\delta
\end{equation}

First we consider the case when $K_0$ does not have components consisting of a single point. Then there is a positive minimum for capacities of components: $\cop\(F_0\)\ge c_0>0$ for any component $F_0$ of $K$.

There is also a positive minimum $2\delta_0$ for distances between different connected components. Then for any $\delta<\delta_0$ the set $(K_0)_\delta$ consists of $s_0\le s$ domains, each containing one connected component of $K_0$; $s_0$ is the number of components of $K_0$.

For any $K\in\mc{K}_s$, $\delta_H\(K_0,K\)<\delta$ and for any connected component $F_0\subset K_0$, the domain $(F_0)_\delta$ contains a compact part $F\subset K$ such that $F_0\subset(F)_\delta$. If $F\subset K$ associated with each component of $K_0$ is a continuum, then  the assertions which we want to prove follow by theorem \ref{thm9.8}. 

In general, we have $F=\bigcup\limits_{j=1}^m F_j$ where $m\le s$ and $F_j$ are connected components of $K$. For $m>1$, we will describe a procedure of connecting components $F_j$ of $F$ into a continuum $\wt{F}$ such that $F\subset\wt{F}\ $, $\delta_H(F,\wt{F})$ and $\cop(\wt{F}\sim F)$ are small. Then we will do the same with each subset --  union of components of $K$ associated with a component of $K_0$; we obtain a compact $\wt{K}\in\mc{K}_s$ with $K\subset\wt{K}$, $\ \delta_H(K,\wt{K})$ is small and $\cop(\wt{K}\sim K)$ is small. Then we prove that the last condition implies that $\left|\mc{E}_\varphi[K]-\mc{E}_\varphi[\wt{K}]\right|$ is small, which will conclude the first part of the proof.

In the process of constructing  $\wt{F}$, we also use the following simple observations. The equilibrium energy functional is decreasing, that is, $K_1\subset K_2$ implies $\mc{E}_\varphi\[K_2\]\le\mc{E}_\varphi\[K_1\]$. We have, therefore, $\mc{E}_\varphi[\ol{\(K_0\)_\delta}]\le\mc{E}_\varphi[K]$. Both compacts $K_0$ and $\ol{\(K_0\)_\delta}$ satisfy conditions of theorem \ref{thm9.8}, which implies $\ \mc{E}_\varphi[\ol{\(K_0\)_\delta}]\to\mc{E}_\varphi[K]\ $ as $\ \delta\to 0$. On the other hand, from $\ \delta_H\(K,K_0\)<\delta$ we obtain $K\subset\(K_0\)_\delta$ and, further, for any $\epsilon>0$
% 9.19
\begin{equation}\label{eq9.17}
\mc{E}_\varphi[K]\ge\mc{E}_\varphi\[K_0\]-\epsilon,
\quad\delta\le\delta(\epsilon).
\end{equation}
Thus, we need only to prove that
% 9.20
\begin{equation}\label{eq9.18}
\mc{E}_\varphi[K]\le\mc{E}_\varphi\[K_0\]+\epsilon,
\quad\delta\le\delta(\epsilon)
\end{equation}
to obtain \eqref{eq9.16}. When proving \eqref{eq9.18}, we may drop subsets of $K$ and pass to smaller compacts $K'\subset K$; indeed, if \eqref{eq9.18} is proved for $K=K'$ then it is proved for $K$ since $\ \mc{E}_\varphi[K]\le\mc{E}_\varphi\[K'\]$.

Now we are ready to describe the construction of continuum $\wt{F}$. We begin with  $F=\bigcup\limits_{j=1}^m F_j\ $, $\ m\le s\ $, $\ F_j$-continua, $\ \delta_H\(F_0,F\)<\delta$.

Suppose first that $\dist\(F_1,F\sim F_1\)\ge2\delta$. Since continuum $F_0$ belongs to $\(F_1\)_\delta\cup\(F\sim F_1\)_\delta$ and sets are disjoint, we conclude that $F_0$ belongs to one of these two sets. Then we can drop another one and pass to $F'\subset F$ which still satisfies $\delta_H\(F',F_0\)<\delta$ and has at least one connected component less than $F$. Then we continue the process with $F'$ in place of $F$ (see remark above).

In the opposite case $\ \dist\(F_1,F\sim F_1\)<2\delta\ $ there exist a point $\ z\ $ in the intersection $\(F_1\)_\delta\cap\(F\sim F_1\)_\delta$. Let $x_1\in F_1$ and $x_2\in F\sim F_1$ are such that $\left|z-x_1\right|\le\delta$ and $\left|z-x_2\right|\le\delta$. Let $\Sigma_1$ be the union of segments $\[z_1,x_2\]$ and $\[z_1,x_2\]$ and $\wt{F}_1=F\cup\Sigma_1$. The number of components of $\wt{F}_1$ is at most $m-1\ $; $\Sigma_1\subset\ol{D_\delta\(z_1\)}$ -- disc of radius $\delta$ and $\delta_H(F,\wt{F}_1)\le\delta$. Thus, $\delta_H(F_0,\wt{F}_1)<2\delta$.

After we repeat this operation $\le m-1$ times, we construct a continuum $\wt{F}=F\cup\Sigma_F$ where $\Sigma_F$ is contained in a union of $\leq m-1$ discs with radii $\leq 2^{m-1}\delta$. We have $\delta_H(F_0,\wt{F})\le2^{m-1}\delta$.
% $\bigcup\limits_{j=1}^{m-1}D_{r_j}\(z_j\)$
We can apply the described procedure to the subset $F\subset K$ associated with each component of $K_0$. It gives us a compact $\wt{K}=K\cup\Sigma_K$ with $\delta_H(K_0,\wt{K})\le2^s\delta$, $\Sigma_K\subset\bigcup\limits_{j=1}^sD_{r_j}\(z_j\)$, $r_j\le2^s\delta$. 

Compact constructed above $\wt{K}$ has as many components as $K_0$ has and components of $\wt{K}$ are $2^s\delta$ close to components of $K_0$; this implies that capacities of components of $\wt{K}$ are $\ge c_0/2>0$ for small enough $\delta$. Thus, by theorem \ref{thm9.8} we have for any $\epsilon>0$
% 9.21
\begin{equation}\label{eq9.19}
\left|\mc{E}_\varphi[\wt{K}]-\mc{E}_\varphi\[K_0\]\right|
\le\epsilon/2,\quad\delta\le\delta(\epsilon)
\end{equation}
($\delta(\epsilon)$ here and in  \eqref{eq9.17} -- \eqref{eq9.18} depend also on $s$, $\ R=\max\limits_{z\in K_0}|z|$, $\ M =max|\varphi|$).

We finish the proof using the following lemma.

% Lemma 9.11
\begin{lemma}\label{lem9.11}
Let $\varphi\in C\(\ol{D}_R\)$, $\wt{K}=K\cup\Sigma\subset\ol{D}_R$ be compacts, $c=\cop(K)>0$. Then for any $\epsilon>0$ there exist $\eta(\epsilon)$ such that
$$
\mc{E}_\varphi[\wt{K}]\le\mc{E}_\varphi[K]
\le\mc{E}_\varphi[\wt{K}]+\epsilon/2\quad\quad\text{for }\quad\cop(\Sigma)<\eta(\epsilon).
$$
\end{lemma}

($\eta(\epsilon)$ depends on $\varphi$, $c$,  $R$, but does not depend on $K$, $\Sigma$.) Proof of this lemma is presented at the end of the section. Now we complete the proof of the theorem 9.10.

We have $\Sigma\subset\bigcup\limits_{j=1}^sD_{r_j}\(z_j\)\subset\ol{D}_R$ (for, say, $R=\max\limits_{z\in K_0}|z|+1$ and $\delta\le 1$), $r_j\le r=2^s\delta$. It is clear that $\cop(\Sigma)\to 0$ as $\delta\to 0$ under these conditions uniformly for $\left\{z_j\right\}$ (actually, we have $\cop(\Sigma)\le CR\delta^{1/s}$ with an absolute constant $C$). Then \eqref{eq9.16} is a combination of \eqref{eq9.19} and lemma \ref{lem9.11}.

To be precise, we mention again that in general our operations with $K$ satisfying $\delta_H\(K_0,K\)<\delta$ yields not $\wt{K}=K\cup\Sigma_K$ with the same condition, but $K'\subset K\cup\Sigma_K$ with $\delta_H\(K_0,K'\)<\delta$ (this $K'$, not $\wt{K}$ satisfies the conditions of theorem \ref{thm9.8}). Then we proceed with lemma \ref{lem9.11} and one-sided inequalities \eqref{eq9.17} and \eqref{eq9.18} as indicated in the comment above related to those inequalities.

It remains to consider the case when $K_0\in\mc{K}_s$ has degenerated components. Let $K_1=K_0+L_0$ where $K_0$ is a union of all non-degenerated continua as above and $L_0$ is a finite number of points in $\mbb{C}\sim K_0$.

We have $\mc{E}_\varphi\[K_1\]=\mc{E}_\varphi\[K_0\]$. Any compact $K\in\mc{K}_s$ with $\delta_H\(K_1,K\)<\delta$ with for small enough $\delta>0$ will have representation $K=K'+L$ where $\delta_H\(L,L_0\)<\delta$, $\delta_H\(K',K_0\)<\delta$. We have in general $\cop(L)>0$ but $L$ is contained in a finite union of discs of radius $=\delta$, so that $\cop(L)$ is small as $\delta\to\infty$. By Lemma \ref{lem9.11}, it follows that $\left|\mc{E}_\varphi\[K'\]-\mc{E}_\varphi[K]\right|$ is also small for small $\delta$. Thus, the situation is reduced to the case of compact $K_0$ without point-components. This concludes the proof of the theorem.
\end{proof}

\begin{proof}[Proof of Lemma \ref{lem9.11}]
We start with a few elementary observations.

For a measure $\mu\in\mc{M}$, $\supp\mu\subset\ol{D}_R$,   $R>\frac12$.   we have

% Equations 9.22, 9.23
\begin{align}\label{eq9.20}
V^\mu(z) &\ge-|\mu|\log(2R),\quad |z|\le R \\
\mc{E}(\mu) &\ge-|\mu|^2\log(2R)
\label{eq9.21}
\end{align}
where $|\mu|=\mu(\mbb{C})$. Equation \eqref{eq9.20} follows by inequality $\log(1/|z-x|)\ge-\log2R$ for $z,x\in\ol{D}_R$; \eqref{eq9.21} follows by \eqref{eq9.20}.

For measures $\lambda=\mu+\nu$ in $\ol{D}_R$ we have
\begin{equation}\label{eq9.22}
\mc{E}(\mu)\le\mc{E}(\lambda)+2|\lambda||\nu|\log(2R).
\end{equation}
Indeed, by \eqref{eq9.20} we have $\mc{E}(\lambda)=\mc{E}(\mu+\nu)=\mc{E}(\mu)+\int V^{2\mu+\nu}d\nu\ge\mc{E}(\mu)-|\nu||2\mu+\nu|\log(2R)$ and \eqref{eq9.22} follows.

Next we will show that for a unit measure $\lambda\in\mc{M}$ in $\ol{D}_R$ and a compact set $\Sigma\subset\ol{D}_R$ we have

% 9.25
\begin{equation}\label{eq9.23}
\lambda(\Sigma)^2\le\frac{\mc{E}(\lambda)
+2\log(2R)}{\gamma(\Sigma)}
\end{equation}
where $\gamma(\Sigma)$ is a Robin constant of $\Sigma$. Denote also by $\omega_\Sigma$ the Robin measure for $\Sigma$. Let $\mu=\left.\lambda\right|_\Sigma$, $\nu=\lambda-\mu$. We have $\mc{E}(\mu/|\mu|)\ge\mc{E}\(\omega_\Sigma\)=\gamma(\Sigma)$. On the other hand, $\mc{E}(\mu/|\mu|)=\mc{E}(\mu)/|\mu|^2=\mc{E}(\mu)/\lambda(\Sigma)^2$. Thus $\lambda(\Sigma)^2\le\mc{E}(\mu)/\gamma(\Sigma)$. Combined with \eqref{eq9.22} it gives \eqref{eq9.23}.

Next, the left inequality in Lemma \ref{lem9.11} is a corollary of monotonicity of $\mc{E}_\varphi[K]$. To prove the one on the right, we consider the equilibrium measure $\lambda$ for $\wt{K}=K\cup\Sigma$. We again set $\mu=\left.\lambda\right|_\Sigma$, $\nu=\lambda-\mu$ and, further, $t=|\mu|=\lambda(\Sigma)$, $r=1/(1-t)$. Since $r\nu$ is a unit measure on $K$ we have
% 9.26
\begin{equation}\label{eq9.24}
\begin{split}
\mc{E}_\varphi[K] &\le\mc{E}_\varphi(r\nu)=r^2\mc{E}_\varphi(\nu)+2r\int\varphi\,d\nu 
=r^2\mc{E}(\lambda-\mu)+2r\int\varphi\,d(\lambda-\mu) \\
&\le r^2(\mc{E}(\lambda)+2t\log(2R))+2r\int\varphi\,d\lambda-2r\int\varphi\,d\mu.
\end{split}
\end{equation}
(To obtain the last inequality, we used \eqref{eq9.22} with $\mu$ and $\nu$ interchanged). We may assume that $t\le 1/2$ (see 9.27 below) then we have $r-1\le 2t$, $r^2-1\le 6t$. Besides, with $\omega=\omega_{\wt{K}}$ we have $\mc{E}_\varphi(\lambda)\le\mc{E}_\varphi(\omega)\le\gamma\(\wt{K}\)+2M$ and further $\mc{E}(\lambda)\le\mc{E}_\varphi(\lambda)-2\int\varphi\,d\lambda\le\gamma\(\wt{K}\)+4M$. Taking into account those inequalities, we obtain from \eqref{eq9.24}
$$
\mc{E}_\varphi[K]\le\mc{E}_\varphi\[\wt{K}\]
+t\(6\gamma\(\wt{K}\)+32M+8\log(2R)\).
$$
At the same time, it follows by \eqref{eq9.23} that
\begin{equation}
t^2=\lambda(\Sigma)^2
\le\frac{\gamma\(\wt{K}\)+4M+2\log(2R)}{\gamma(\Sigma)}.
\end{equation}
If $\cop(\Sigma)$ is small, then $\gamma(\Sigma)$ is large uniformly over all $\Sigma\subset D_R$. At the same time, $\cop\(\wt{K}\)\ge c>0$ implies that $\gamma\(\wt{K}\)\le 1/\log\frac1c$ so that $t$ is small and $\mc{E}_\varphi[K]-\mc{E}_\varphi\[\wt{K}\]$ is also small uniformly for $K, \ \Sigma\subset\ol{D}_R, \ \cop(K)\ge c>0$. Lemma \ref{lem9.11} is proven.
\end{proof}

% Sec. 9.8
\subsection{Continuity of $\mc{E}_\varphi$ for bounded $\varphi\in C^\ast$}\label{subsec9.8}

We denote by the $C^\ast=C^\ast(\mbb{C})$ set of real-valued functions $\varphi$ continuous in  $\mbb{C}\sim e$ 
for some finite set $e\subset\mbb{C}$ (depending on $\varphi$). Here we prove that the continuity of $\mc{E}_\varphi[K]$ on $\mc{K}_s$ which is asserted by theorem 9.10 for $\varphi\in C(\mbb{C})$ remains valid for $\varphi\in C^\ast$ under the additional assumption of boundedness of $\varphi$.
% Lemma 9.12
\begin{lemma}\label{lem9.12}
Let $\varphi\in C^\ast$ and $|\varphi|\le M$ in $\mbb{C}\sim e$. Then $\mc{E}_\varphi[K]:\mc{K}_s\to(-\infty,\infty)$ is continuous.
\end{lemma}

In order to prove Lemma \ref{lem9.12}, we need the following two auxiliary assertions. First of them is a folklore

% Lemma 9.13
\begin{lemma}\label{lem9.13}
For a fixed $R\ge 1$ there is a constant $C=C(R)$ such that for any two compacts $E,F\in\ol{D}_R$ we have $\cop(E\cup F)\le C(\cop(E)+\cop(F))$.
\end{lemma}

% Lemma 9.14
\begin{lemma}\label{lem9.14}
For fixed $s\in\mbb{N}$, $R>0$ and any $\epsilon>0$, there exist $\delta(\epsilon)=\delta(\epsilon,s,R)>0$ such that for any two compacts $E,F\in\ol{D}_R$ we have
$$
|\cop(E)-\cop(F)|<\epsilon\quad\text{for }\quad\delta_H(E,F)<\delta(\epsilon).
$$
\end{lemma}

To prove the lemma \ref{lem9.14} it is enough to note the foollowing. If both capacities are small the assertion of the lemma is trivial. If one of them is bounded away from zero then the assertion is a particular case of theorem \ref{thm9.10} with $\varphi = 0 $. Now, we turn to the proof of  lemma \ref{lem9.12}

For a fixed $\wt{K}\in\mc{K}_s$ we need to show that functional $\mc{E}_\varphi$ is continuous at $\wt{K}$. Let $\gamma=\gamma\(\wt{K}\)$ (it is a finite number and $c=\cop\(\wt{K}\)=e^{-\gamma}>0$), let $R=\max\limits_{z\in\wt{K}}|z|+1$.

Next, let $e=\left\{a_1,\dotsc,a_p\right\}$; define $\mc{D}_r=\bigcup\limits_{j=1}^PD_r\(a_j\)$ - union of discs  $D_r\(a_j\)$ of radius $r$ centered at $a_j$. For $r\in(0,1]$, we define the function $\varphi_r(z)$ be continuous in $\mbb{C}$, harmonic in $\mc{D}_r$, and $\varphi_r=\varphi$ in $\mbb{C}\sim\mc{D}_r$; thus, $\varphi_r(z)$ is a regularization of $\varphi$ near singular points.
%(If discs $D_r\(a_j\)$ are disjoint then in each such disc $\varphi_r$ is the solution of the Dirichlet problem with boundary values $\varphi$.)

Let $\lambda$ be the equilibrium measure for $\wt{K}$ in the field $\varphi$ and $\lambda^r$ be the equilibrium measure of the same compact in the field $\varphi_r$. We have
$$
\mc{E}_\varphi\[\wt{K}\] \le\mc{E}_\varphi\(\lambda^r\)=\mc{E}\(\lambda^r\)+2\int\varphi\,d\lambda^r 
=\mc{E}_{\varphi_r}\(\lambda^r\)+2\int\(\varphi-\varphi_r\)d\lambda^r.
$$
%\begin{equation*}
%\begin{split}
%\mc{E}_\varphi\[\wt{K}\] &\le\mc{E}_\varphi\(\lambda^r\)=\mc{E}\(\lambda^r\)+2\int\varphi\,d\lambda^r \\
%&=\mc{E}_{\varphi_r}\(\lambda^r\)+2\int\(\varphi-\varphi_r\)d\lambda^r.
%\end{split}
%\end{equation*}
From here (note also that we can interchange $\varphi$ and $\varphi_r$) 
$$\left|\mc{E}_\varphi\[\wt{K}\]-\mc{E}_{\varphi_r}\[\wt{K}\]\right|\le 4M\lambda^r\(\mc{D}_r\)$$ 
Further, by \eqref{eq9.23} with $\Sigma=\ol{\mc D}_r$ we have as $r\to 0$
% 9.28
\begin{equation}\label{eq9.24.1}
\lambda^r\(\ol{\mc D}_r\)\le\(\gamma(\wt{K})+4M+2\log(2R)\)/\gamma\(\ol{\mc D}_r\)\to 0.
\end{equation}
Combined with the inequality above, it makes
% 9.29
\begin{equation}
\label{eq9.25}
\left|\mc{E}_\varphi(\wt{K})-\mc{E}_{\varphi_r}(\wt{K})\right|\to 0\quad\text{as }r\to 0.
\end{equation}
The same arguments show that this relation is also valid for any compact $K\in\mc{K}_s$ with $\delta_H(K,\wt{K})<\delta$ if $\delta$ is small enough
% 9.30
\begin{equation}
\label{eq9.26}
\left|\mc{E}_\varphi(K)-\mc{E}_{\varphi_r}(K)\right|\to 0\quad\text{as }r\to 0
\end{equation}
and, moreover, this limit is uniform with respect to $K$ with the indicated properties. Indeed, according to \eqref{eq9.24.1} the left-hand side in \eqref{eq9.25} is bounded by
$$
4M(\gamma(K)+4M+2\log(2R))/\gamma\(\ol{\mc D}_r\)
$$
which is uniformly small as $\gamma\(\ol{\mc{D}}_r\)$ is large since $\gamma(K)$ is bounded by Lemma \ref{lem9.13}.  Now, assertion of the lemma      \ref{lem9.12}  follows by \eqref{eq9.25} and \eqref{eq9.26} in combination with theorem \ref{thm9.10}. 

% Sec. 9.9
\subsection{Semi-continuity of $\mc{E}_\varphi$ for $\varphi\in C^\ast$}\label{subsec9.9}

First, we prove that the assertion of lemma \ref{lem9.12} remains valid even if we remove the the assumption that $\varphi$ is bounded from above.
\begin{lemma}\label{lem9.15}
Let $\varphi\in C^*$ and $\varphi(x)\geq -M$ in $\mbb C\sim e$. Then $\mc E_\varphi$ is continuous on $\mc K_s$.
\end{lemma}
\begin{proof}
Again, for a fixed $\tilde{K}\in\mc K_s$ we need to show that $\mc E_\varphi$ is continuous at $\tilde{K}$. We have $c=$cap$(\tilde{K})=e^{-\gamma}>0$. By lemma \ref{lem9.14} there is $\delta_0>0$ such that for any $K\in\mc K_s$, $\delta_H(K,\tilde{\mc K})<\delta\leq\delta_0\leq 1$ we have cap$(K)\geq c/2$ or $\gamma(K)\leq\gamma+\ln 2$. We also set $R=1+\max\limits_{z\in\tilde{K}}|z|$ then $K\in D_R$ for any $K\in\mc K$ with $\delta_H(K,\tilde{\mc K})\leq\delta_0$.
%(that is $K\subset(\tilde{K})_\delta_0$).

For any such compact $K$ we may apply \ref{eq9.7}, \ref{eq9.7.1} in lemma \ref{lem9.3} to the equilibrium measure $\lambda=\lambda_{\varphi,K}$ in place of $\mu$. We need to take nto accaunt that formulas  \ref{eq9.7}, \ref{eq9.7.1}  were derived for the case $K\in D_{1/2}$. To reduce the current case to the one in lemma \ref{lem9.3} we make substitution $\zeta=z/2R$ to map a neighbourhood of $\tilde K$ onto class of compacts in $\ol{D}_{1/2}$.
Then we come to the following: for
$$
M_0=\ \inf\limits_{\epsilon>0}\ \bigg(\max\limits_{z\in K\sim\mc D_{R_\epsilon}}\varphi(z)+2\gamma(K\sim\mc D_\epsilon)+2\log 2R\bigg)\ +1
$$
we have $\lambda\big(\{z\in K:\varphi(z)>M_0\}\big)=0$. The only term depending of $K$ in the formula for $M_0$ is $\gamma(K\sim\mc D_\epsilon)$. By lemmas \ref{lem9.13} and \ref{lem9.14} the last expression is bounded for all $K\subset{(\tilde{K})}_{\delta_0}$ if $\epsilon$ is small enough. Thus, we can find $M_1$ such that $\lambda(\{\varphi\geq M_1\}\cap K)=0$ for any   $K\subset{(\tilde{K})}_{\delta_0}$.

The, for $\tilde{\varphi}(z)=\min\{\varphi(z),M_1\}$ we have $\mc E_\varphi[K]=\mc E_{\tilde{\varphi}}[K]$ for any $K\in(\tilde{K})_{\delta_0}$. But $\tilde{\varphi}$ is bounded from above and from below and lemma \ref{lem9.15} follows by lemma \ref{lem9.12}.
\end{proof}

Now, we finish the proof of theorem \ref{thm3.2}. Let $\varphi\in C^*$. Define $\varphi_n=\max\{\varphi,-n\}$. $\mc E_{\varphi_n}$ is continuous by lemma \ref{lem9.15}. On the other hand we have $\mc E_{\varphi_n}[K]\leq\mc E_{\varphi_{n+1}}[K]$ for $K\in\mc K_s$. Thus $\mc E_{\varphi_n}$ is monotone decreasing sequence of continuous functions on $\mc K_s$ and its limit is upper semicontinuous. Theorem \ref{thm3.2} is proven.

% Sec.9.10
\subsection{Proof of  Lemma \ref{lem3.5} and Teorem \ref{thm3.4} }\label{subsec9.10} Proof of lemma \ref{lem3.5} is based on the following general fact: first varition of equilibrium energy of a compact coinside with the first variation of energy of equilibrium measure. We reproduce here the proof from \cite{PeRa94} where case $\varphi =0$ was considered. Presence of external field does not  essentially affect arguments used in the proof
(taking into account remark \ref{rem4.1} we can assume that functions $h$ in variations are harmonic in a neighbourhood of $K$; then the external energy term is analytic in $t$). 
 
Let $K\in\mc{K}_s$ be a compact in $\mbb{C}$, $\varphi$ is harmonic in $\Omega\sim e$ where $K\subset\Omega$ and $e$ is a finite set. We will consider a variation $z\to z^t=z+th(z)$ associated with a smooth function $h$ with $h(z)=0$ for $z\in(e)_\delta$ with some small enough $\delta>0$.
Let $\lambda=\lambda_{\varphi,K}$ and $\mu_t$ be the equilibrium measure for $K^t=\left\{z^t,z\in K\right\}$ in the field $\varphi$. For small enough $t>0$ mapping $z\to z^t$ is one-to-one, so that there exist the inverse mapping $z^{-t}=z-th(z)+ o(t)$, that is $\(z^{-t}\)^t=\(z^t\)^{-t}=z$. Let $\mu^{-t}=\(\mu_t\)^{-t}\in\mc{M}(K)$. We define $d=D_h\mc{E}_\varphi(\lambda)$, $d(t)=D_h\mc{E}_\varphi\(\mu_t\)$ (see  \eqref{eq4.2} and \eqref{eq4.3}). 
\begin{gather*}
\mc{E}_\varphi\[K^t\]\le\mathcal E_\varphi\(\lambda^t\)
=\mc{E}_\varphi(\lambda)+td+O(t^2)=\mc{E}_\varphi[K]+td+O(t^2) \\
\mc{E}_\varphi[K]\le\mc{E}_\varphi\(\mu^{-t}\)
=\mc{E}_\varphi\(\mu_t\)-td(t)+O(t^2)=\mc{E}_\varphi[K^t]-td(t)+O(t^2).
\end{gather*}
It is not difficult to verify that constants in all  $O(t^2)$ are explicit and uniform for small enough $t$. From here
$$
td(t)+O(t^2)\le\mc{E}_\varphi\[K^t\]-\mc{E}_\varphi[K]\le td+O(t^2)
$$
It implies, in particular, that $\mu_t$ and $\lambda$ are close in energy metric as $t\to 0$, so that $d(t)\to d$ as $t\to 0$ and, subsequently that
\begin{equation}
\label{eq9.27}
\lim_{t\to 0+}\frac1t\(\mc{E}_\varphi\[K^t\]-\mc{E}_\varphi[K]\)
=\lim_{t\to 0+}\frac1t\(\mc{E}_p\(\lambda^t\)-\mc{E}_\varphi(\lambda)\).
\end{equation}

It is important to observe now that the assertion \eqref{eq9.27} remains valid if $K$ is a compact in the extended complex plane under the assumption that the point at infinity is included in the set  $e$  of singular points of $\varphi$.  Then the condition $h(z)=0$ for $z\in(e)_\delta$ means that the variation leaves a neighbourhood of infinity fixed and all the arguments above remain valid. Now we can complete proofs.

We return to a class $\mc{T}$ of compact satisfying conditions of Theorem \ref{thm3.1}. There exist a maximizing sequence $\Gamma_n$ . Next we select a subsequence from $\left\{\Gamma_n\right\}$ which converges in spherical  $\delta_H$-metric to $S$. By condition (iv) there is a disc $D$ such that $\Gamma_n\sim D\in\mc{T}$ and, therefore, converges to $S\sim D$. Since $\mc{E}_\varphi\[\Gamma_n\sim D\]\ge\mc{E}_\varphi\[\Gamma_n\]$ sequence $\Gamma_n\sim D$ is also maximizing. Since $\sup\limits_{\Gamma}\mc{E}_\varphi[\Gamma]$ is finite the sequence of corresponding equilibrium measures $\lambda_n$ converges in energy to measure $\lambda$. By lemma \ref{lem9.1} we can reduce situation to the case where $\Gamma_n$ and $S$ belong, say,  to the disc $D_{\frac12}$. By Theorem \ref{thm3.2} we have $\mc{E}_\varphi[S]=\sup\limits_{\Gamma}\mc{E}_\varphi[\Gamma]$,  $\lambda$  is equilibrium mesure for  $S$.

It follows that for any smooth $h$ vanishing in $(e)_\delta$ with a positive $\delta$ we have $\lim\limits_{t\to 0+}\frac1t\(\mc{E}_\varphi[S^t]-\mc{E}_\varphi[S]\)=0$. From here and \ref{eq9.27} we obtain that $D_h\mc{E}_\varphi(\lambda)=0$ for any such $h$.  This implies that $\lambda$ is $(A,\varphi)$-critical and, therefore, its weighted potential  has $S$-property. 



%\bibliographystyle{amsplain}
%\bibliography{S-curves-9}

\def\cprime{$'$}
\providecommand{\bysame}{\leavevmode\hbox to3em{\hrulefill}\thinspace}
\providecommand{\MR}{\relax\ifhmode\unskip\space\fi MR }
\providecommand{\MRhref}[2]{%
  \href{http://www.ams.org/mathscinet-getitem?mr=#1}{#2}
}
\providecommand{\href}[2]{#2}


\end{document}